\documentclass{amsart}
\usepackage[latin9]{inputenc}
\usepackage{textcomp}
\usepackage{amsthm}
\usepackage{amstext}
\usepackage{amssymb}
\usepackage{esint}

\makeatletter
\numberwithin{equation}{section}
\numberwithin{figure}{section}
\theoremstyle{plain}
\newtheorem{thm}{\protect\theoremname}
  \theoremstyle{remark}
  \newtheorem{rem}[thm]{\protect\remarkname}
  \theoremstyle{definition}
  \newtheorem{defn}[thm]{\protect\definitionname}
  \theoremstyle{plain}
  \newtheorem{prop}[thm]{\protect\propositionname}
  \theoremstyle{plain}
  \newtheorem{lem}[thm]{\protect\lemmaname}

\usepackage{amsthm}\usepackage{amstext}\usepackage{amsfonts}\usepackage{bbm}\usepackage{mdef}\usepackage{accents}

\numberwithin{equation}{section}
\numberwithin{figure}{section}

\email{gess@math.tu-berlin.de}

\thanks{{\bf Acknowledgements:}  The author would like to thank Michael Röckner for valuable discussion and comments}

  \providecommand{\lemmaname}{Lemma}
  \providecommand{\remarkname}{Remark}
\providecommand{\theoremname}{Theorem}

\makeatother

  \providecommand{\definitionname}{Definition}
  \providecommand{\lemmaname}{Lemma}
  \providecommand{\propositionname}{Proposition}
  \providecommand{\remarkname}{Remark}
\providecommand{\theoremname}{Theorem}

\begin{document}

\title[Finite time extinction for stochastic sign fast diffusion]{Finite time extinction for stochastic sign fast diffusion and self-organized criticality}

\author{Benjamin Gess}

\address{Benjamin Gess \\
Institut für Mathematik, Humboldt-Universität zu Berlin, Germany}
\begin{abstract}
We prove finite time extinction for stochastic sign fast diffusion equations driven by linear multiplicative space-time noise, corresponding to the Bak-Tang-Wiesenfeld model for self-organized criticality. This solves a problem posed and left open in several works: \cite{B13,RW11,BDPR12,BR12b,BDPR09-2,BDPR09-3}. The highly singular-degenerate nature of the drift in interplay with the stochastic perturbation causes the need for new methods in the analysis of mass diffusion and several new estimates and techniques are introduced.
\end{abstract}

\date{\today}

\maketitle

\section{Introduction}

Self-organized criticality (SOC) is a model of complex behavior that has attracted much attention in physics (cf. \cite{BTW88,Z89,BI92,J98,T99-2,CCGS90,DG94,GDG90} among many others). We recall from \cite{BI92}: \textit{The term ``criticality'' refers to the power-law behavior of the spatial and temporal distributions, characteristic of critical phenomena. ``Self-organized'' refers to the fact that these systems naturally evolve into a critical state without any tuning of the external parameters, i.e. the critical state is an attractor of the dynamics. }It is this robust tendency to evolve into a critical state that distinguishes SOC from more classical models of criticality as for example observed in phase-transitions. 

Based on a cellular automaton algorithm, in \cite{BI92} a continuum limit related to the original sand pile model introduced by Bak-Tang-Wiesenfeld (BTW) in \cite{BTW88} was derived, leading to a highly singular-degenerate PDE of the type
\begin{align}
\partial_{t}Z_{t} & \in\D H(Z_{t}-z_{c}),\quad\text{on }[0,T]\times\mcO\label{eq:BTW}\\
0 & \in H(Z_{t}-z_{c}),\quad\text{on }\partial\mcO,\nonumber 
\end{align}
where $H$ is the Heaviside function, $z_{c}$ is the critical state and $\mcO\subseteq\R^{d}$ is a bounded, smooth domain. Rewriting \eqref{eq:BTW} as an equation for $X_{t}=Z_{t}-z_{c}$ leads to 
\begin{align}
\partial_{t}X_{t} & \in\D H(X_{t}),\quad\text{on }[0,T]\times\mcO\label{eq:det_BTW}\\
0 & \in H(X_{t}),\quad\text{on }\partial\mcO.\nonumber 
\end{align}
The effect of robust evolution/relaxation in finite time into a subcritical state can now be recast as finite time extinction of $(X_{t})^{+}$, i.e. $X_{t}\le0$ after some finite time $\tau_{0}$. If we restrict to the relaxation of purely supercritical states (i.e. $Z_{0}\ge z^{c}$ resp. $X_{0}\ge0$) then the relaxation into the critical state corresponds to the extinction of $X_{t}$ in finite time, i.e. $X_{t}\equiv0$ after some finite time $\tau_{0}$. 

As regarding finite time extinction for the related fast diffusion equation, i.e. for
\begin{align}
\partial_{t}X_{t} & \in\D(|X_{t}|^{m-1}X_{t}),\quad\text{on }[0,T]\times\mcO\label{eq:FDE}\\
X_{t} & =0,\quad\text{on }\partial\mcO,\nonumber 
\end{align}
with $m\in(0,1)$ a thorough analysis may be found in \cite{V06} and the references therein. Note that \eqref{eq:BTW} corresponds to to $m\downarrow0$ in \eqref{eq:FDE}. Starting from this, several generalizations have been obtained. For example, recently finite time extinction for fractional fast diffusion equations of the type
\begin{align*}
\partial_{t}X_{t} & =-(-\D)^{\s/2}(|X_{t}|^{m-1}X_{t}),\quad\text{on }[0,T]\times\mcO\\
X_{t} & =0,\quad\text{on }\partial\mcO,
\end{align*}
with $\s\in(0,2)$, $m\in(0,1)$ has been shown in \cite{dPQRV12}. The question of existence of a non-trivial continuation after the extinction time for fast diffusion equations with sink has been solved in \cite{GV97}. In the case of (fractional) fast diffusion equations energy inequalities for $L^{p}$-norms, choosing $p$ large enough, may be used to prove finite time extinction. As we will point out in detail in Section \ref{sub:proof-outline} below, in case of \eqref{eq:det_BTW} this ceases to be true and one has to work with the $L^{1}$-norm instead, thus causing the situation of \eqref{eq:det_BTW} to be quite different from \eqref{eq:FDE}. Finite time extinction for \eqref{eq:det_BTW} has been proven for the first time in \cite{DD79}. In fact, more general equations of the type
\[
\partial_{t}X_{t}\in\D\phi(X_{t})
\]
are treated in \cite{DD79} and a sufficient (assuming $\phi$ to be maximal monotone) and necessary (if $\phi$ is continuous) condition on $\phi$ for finite time extinction is proven. Very recently, an alternative proof of finite time extinction for \eqref{eq:det_BTW}\textit{ }has been given in \cite{B13}. In one spatial dimension a detailed analysis of the dynamics of the total variation flow and thus the sign fast diffusion has been developed in \cite{BF12}. These results are complemented by the present paper by proving finite time extinction for stochastically perturbed versions of \eqref{eq:det_BTW}.

As it has been pointed out in \cite{DG94,GDG98,DG92} it is more realistic to include stochastic perturbations in \eqref{eq:BTW} modeling the energy randomly added to the system, accounting for the removed microscopic degrees of freedom in the continuum limit and reflecting model uncertainty. As pointed out above, the robustness of self-organization in SOC is crucial. Based on this, the question arises whether this robustness with respect to perturbations is actually satisfied by \eqref{eq:BTW}, again leading to the study of stochastically perturbed versions of \eqref{eq:BTW}. Generally speaking, the resulting equations are stochastic partial differential equations (SPDE) of the following type
\begin{align*}
dX_{t} & \in\D H(X_{t})dt+B(X_{t})dW_{t},\quad\text{on }[0,T]\times\mcO\\
0 & \in H(X_{t}),\quad\text{on }\partial\mcO,
\end{align*}
where $B$ are suitable diffusion coefficients. Particular attention (cf. e.g. \cite{RW11,BDPR12,BR12b,BDPR09-2} among others) has been paid to the case of linear multiplicative space-time noise, i.e. to
\begin{align}
dX_{t} & \in\D H(X_{t})dt+\sum_{k=1}^{N}f_{k}X_{t}d\b_{t}^{k},\quad\text{on }[0,T]\times\mcO\label{eq:stoch_BTW}\\
0 & \in H(X_{t}),\quad\text{on }\partial\mcO,\nonumber 
\end{align}
where $X_{0}\ge0$, $N\in\N$, $f=(f_{k})_{k=1,\dots,N}\in C^{2}(\mcO;\R^{N})$ and $\b=(\b^{k})_{k=1,\dots,N}$ is a standard Brownian motion in $\R^{N}$. Again, the key property of robust relaxation of supercritical states ($X_{0}\ge0$) into subcritical ones can be (re-)stated as the problem of finite time extinction: Let
\[
\tau_{0}=\inf\{t\ge0|\ X_{t}(\xi)=0\text{ for a.e. }\xi\in\mcO\}.
\]
Finite time extinction can then be stated as $\P[\tau_{0}<\infty]=1$ for all nonnegative initial values $X_{0}=x\ge0$. 

Despite its fundamental nature, the question of finite time extinction for the stochastic BTW model with linear multiplicative space-time noise \eqref{eq:stoch_BTW} has remained an open problem for several years. The mathematical difficulty of an analysis of the diffusion of mass and finite time extinction for \eqref{eq:stoch_BTW} stems from the highly singular-degenerate nature of the drift $\D H$ and its interplay with the stochastic perturbation. For example, the problem of finite time extinction for \eqref{eq:stoch_BTW} has been posed and left as an open problem in the works \cite{B13,RW11,BDPR12,BR12b,BDPR09-2,BDPR09-3}. The main purpose of the present paper is to resolve this issue by proving finite time extinction for \eqref{eq:stoch_BTW}, without any restriction on the dimension $d$ of the underlying domain $\mcO\subseteq\R^{d}$.

In order to develop a finer analysis of the diffusion of mass for \eqref{eq:stoch_BTW} it turns out to be crucial to work in a pathwise setting, i.e. we base our analysis on a transformation of \eqref{eq:stoch_BTW} into a random PDE which in turn may be analyzed for each fixed Brownian path $t\mapsto\b_{t}(\o)$. In the above mentioned works weaker results proving finite time extinction only with positive probability could be obtained. That is, it could be shown that the measure of Brownian paths for which finite time extinction occurs is non-zero (cf. Section \ref{sub:review} below). The pathwise approach pursued in this paper allows a detailed understanding of the relation between the behavior of Brownian paths and finite time extinction. This leads to a better understanding why so far only finite time extinction with positive probability could be shown and finally leads to a proof of finite time extinction $\P$-almost surely.

\subsection{Overview of known results\label{sub:review}}

While finite time extinction for the stochastic BTW model could not be proven so far, important progress concerning the (stochastic) Zhang model, i.e. for 
\begin{align*}
dX_{t} & \in\D(H(X_{t})(1+\d X_{t}))dt+B(X_{t})dW_{t},\quad\text{on }[0,T]\times\mcO\\
0 & \in H(X_{t}),\quad\text{on }\partial\mcO,
\end{align*}
with $\d>0$ and partial results for the BTW case have been obtained in recent years. Before giving a short overview of these results we will point out a key mathematical difference between the Zhang and the BTW model. 

We (informally) compute
\[
\D\phi(X)=\div\left(\phi'(X)\nabla X\right)=\phi'(X)\D X+\phi''(X)|\nabla X|^{2},
\]
where 
\begin{equation}
\phi(r)=\begin{cases}
1+\d r & ,\text{ if }r>0\\
{}[0,1] & ,\text{ if }r=0\\
0 & \text{, if }r<0
\end{cases}\label{eq:SOC_nonlinearity-1}
\end{equation}
with $\d=0$ in the BTW, $\d>0$ in the Zhang model. Since we are dealing with nonlinearities being singular at zero (cf. \eqref{eq:SOC_nonlinearity-1}) the coercivity coefficient $\phi'(X)$ is singular at zero thus causing fast diffusion of mass for small values of $X$. As we will see below, this singularity is responsible for the effect of finite time extinction. On the other hand $\phi'(X)$ may degenerate for large values of $X$ making it difficult to control the diffusion of mass when $X$ is large. While for fast diffusion equations (FDE)
\[
\phi(r)=r^{m}\sgn(r),\ m\in[0,1),
\]
and the Zhang model the diffusion coefficient $\phi'(r)$ is non-degenerate at least locally in $r$, the BTW model ($\phi'(r)=\delta_{0}$) is highly degenerate making the analysis of mass diffusion and thus the proof of finite time extinction much harder. On the other hand, we note that the arguments presented in this paper depend on the simple structure of the nonlinearity in the BTW model ($\phi=H$) and the methods do not seem to directly extend to fast diffusion equations.

We will now give a brief overview of the known results concerning finite time extinction for the stochastic BTW and Zhang model. Existence and uniqueness of solutions to multivalued SPDE of the type%
\footnote{In fact, in \cite{BDPR09-2} the diffusion coefficients $f_{k}$ were supposed to be of the special form $\mu_{k}e_{k}$ with $\mu_{k}\in\R$ and $e_{k}$ being eigenvalues of $-\D$. However, this does not seem to be crucial for the methods developed in \cite{BDPR09-2}.%
}
\begin{eqnarray}
dX_{t} & \in & \D\phi(X_{t})dt+\sum_{k=1}^{N}f_{k}X_{t}d\b_{t}^{k},\quad\text{on }[0,T]\times\mcO\label{eq:stoch_SOC-1}\\
0 & \in & \phi(X_{t}),\quad\text{on }\partial\mcO,\nonumber 
\end{eqnarray}
with $f_{k}\in H_{0}^{1}(\mcO)$ being sufficiently smooth and $\phi:\R\to2^{\R}$ being a maximal monotone, multivalued function satisfying a polynomial growth condition, has been first shown in \cite{BDPR09-2} in dimension $d\le3$. This includes FDE, the Zhang model and the BTW model. As a further result, positivity preservation (i.e. $X_{t}\ge0$ if $x_{0}\ge0$) has been proved in \cite{BDPR09-2}.

We define 
\[
\tau_{0}(\o):=\inf\{t\text{\ensuremath{\ge}}0|X_{t}(\o)=0,\ \text{a.e. in }\mcO\}.
\]
By a supermartingale argument it has been proved in \cite{BDPR12} that $X_{t}=0$, $d\xi$-a.e. for all $t\ge\tau_{0}$, $\P$-almost surely. This also follows from the results given in Section \ref{sec:exp_decay} below. As concerning finite time extinction we distinguish the following concepts:
\begin{description}
\item [{\textmd{(F1)}}] Extinction with positive probability for small initial conditions: $\P[\tau_{0}<\infty]>0$, for small $X_{0}=x_{0}$.
\item [{\textmd{(F2)}}] Extinction with positive probability: $\P[\tau_{0}<\infty]>0$, for all $X_{0}=x_{0}$.
\item [{\textmd{(F3)}}] Finite time extinction: $\P[\tau_{0}<\infty]=1$, for all $X_{0}=x_{0}$.
\end{description}
While from a mathematical viewpoint also the (weaker) properties (F1), (F2) are interesting, the robustness of the relaxation into subcritical states in SOC is fundamental in physics and thus mainly (F3) is relevant from the SOC point of view.

In order to prove (F1), in \cite{BDPR09-2} some coercivity/non-degeneracy of the diffusion had to be assumed, i.e. $\phi'\ge\d>0$. As applied to SOC this corresponds to restricting to the Zhang model. Under this assumption and restricting to $\mcO=[\text{0,\ensuremath{\pi}}]$, (F1) has been shown in \cite{BDPR09-2}. In the subsequent work \cite{BDPR09-3}, for FDE the restriction to one space dimension was relaxed to $d\in\N$ as long as $m\in[\frac{d-2}{d+2},1).$ 

More recently, the BTW model was considered in \cite{BR12b} for $d\le3$ where \textit{asymptotic extinction} was shown, i.e. 
\[
\int_{0}^{\infty}|\mcO\setminus\mcO_{0}^{t}|dt<\infty,\quad\P\text{-a.s.},
\]
where $|\cdot|$ is the Lebesgue measure and $\mcO_{0}^{t}=\{\xi\in\mcO|X_{t}(\xi)=0\}$. Note that (F3) implies asymptotic extinction. Moreover, assuming a non-degeneracy condition for the noise (i.e. $\sum_{k=1}^{N}f_{k}^{2}>0$) an exponential decay property of $X_{t}$ was shown (cf. also Section \ref{sec:exp_decay} below, where this result is improved).

The survey article \cite{BDPR12} revisits the results obtained in \cite{BDPR09-2,BDPR09-3,BR12b} and some technical assumptions are relaxed. In particular, the non-degeneracy condition on $\phi$ required in \cite{BDPR09-2} is dropped, thus proving (F1) for the BTW model for $d=1$. 

In \cite{RW11} a general class of processes $X_{t}$ is analyzed, merely satisfying a certain energy inequality and extinction properties for such $X_{t}$ are shown. Applied to equations of type \eqref{eq:stoch_SOC-1} this allows for several generalizations, e.g. replacing the Laplacian $\D$ by its fractional powers $-(-\D)^{\a}$ with $\a\in(0,1)$ as also studied in \cite{dPQRV12}. Concerning SOC, one of the main results obtained in \cite{RW11} is (F3) for the Zhang model, while for the BTW model still only (F1) for $d=1$ could be shown. 

If we do not insist on nonnegativity of solutions it makes sense to consider random perturbations of additive type, i.e.
\begin{align}
dX_{t} & \in\D\phi(X_{t})dt+dW_{t},\quad\text{on }[0,T]\times\mcO\label{eq:SOC_additive_noise}\\
0 & \in\phi(X_{t}),\quad\text{on }\partial\mcO,\nonumber 
\end{align}
where $W_{t}$ is an appropriate Wiener process. In fact, this additive type of noise has been suggested in the physics literature (cf. e.g. \cite{DG94,GDG98}). SPDE of the form \eqref{eq:SOC_additive_noise} (actually also allowing more general, multiplicative noise $B(X_{t})dW_{t}$) have been considered in \cite{GT11} where the existence and uniqueness of solutions (for all $d\in\N$) as well as ergodicity (for $d=1$ and additive noise) has been shown. Based on the results developed in \cite{G13-4} one may expect that this implies the existence of a random attractor consisting of a single random point, which we expect to prove in subsequent work.

At last we should mention the very recent work \cite{BR13} on the related stochastic total variation flow, where (F1) has been shown in dimensions $d\le3$. 

In conclusion, despite the large amount of works addressing finite time extinction for the stochastic BTW model it has remained an open question up to now whether this happens with probability one. In this paper we solve this problem by proving (F3) for the stochastic BTW model with underlying bounded domain $\mcO\subseteq\R^{d}$ for all $d\ge1$.

\subsection{Main result and outline of the proof}

We will now state our main result in more detail and give a brief, informal overview of our approach. In the following let $\mcO\subseteq\R^{d}$ be an open domain with smooth boundary $\partial\mcO$, $f=(f_{k})_{k=1,\dots,N}\in C^{2}(\mcO;\R^{N})$ and $\b=(\b^{k})_{k=1,\dots,N}$ be a standard Brownian motion in $\R^{N}$. As above, we restrict to nonnegative initial conditions (and thus to nonnegative solutions) so that the stochastic BTW model may equivalently be written as
\begin{align}
dX_{t} & \in\D\sgn(X_{t})dt+\sum_{k=1}^{N}f_{k}X_{t}d\b_{t}^{k},\quad\text{on }[0,T]\times\mcO\label{eq:BTW_SOC}\\
0 & \in\sgn(X_{t}),\quad\text{on }\partial\mcO,\nonumber 
\end{align}
with $X_{0}=x_{0}$, where $\sgn$ denotes the maximal monotone extension of the sign function. 

We set $\mu_{t}:=-f\cdot\b_{t}=-\sum_{k=1}^{N}f_{k}\b_{t}^{k}$, $\td\mu:=\frac{1}{2}|f|^{2}=\frac{1}{2}\sum_{k=1}^{N}f_{k}^{2}$ and we consider the transformation $Y_{t}:=e^{\mu_{t}}X_{t}$. An informal calculation shows
\begin{equation}
\partial_{t}Y_{t}\in e^{\mu_{t}}\D\sgn(Y_{t})-\td\mu Y_{t}.\label{eq:TPME-intro-1}
\end{equation}
The analysis of \eqref{eq:BTW_SOC} presented in this paper will be essentially based on an analysis of \eqref{eq:TPME-intro-1}. A rigorous justification of this transformation will be given in Section \ref{sec:tranformation} below. 

As we will see in Section \ref{sec:FTE}, a mild condition on the decay of the mass of the level sets of $\td\mu$ (e.g. $|\{\xi\in\mcO|0<\td\mu(\xi)<\ve\}|\lesssim\ve^{\d}$ for all $\ve>0$ small enough and some $\d>0$) implies
\begin{enumerate}
\item [(H)]: For all $p\ge1$ there is a $t_{0}=t_{0}(p,\o)$ such that 
\[
\int_{\mcO}e^{p(-\mu_{t}-\td\mu t)}d\xi=\int_{\mcO}e^{p\sum_{k=1}^{N}f_{k}(\xi)\b_{t}^{k}-\frac{1}{2}f_{k}^{2}(\xi)t}d\xi\le C(p,\o)<\infty,
\]
for $\P$-a.a $\o\in\O$ and all $t\ge t_{0}$. 
\end{enumerate}
Note that, in particular, the cases of vanishing noise ($\td\mu\equiv0$) and full noise ($\td\mu>0$) are trivially covered by the above condition. 

Roughly speaking, our main result is
\begin{thm}[Finite time extinction]
\label{thm:main-intro}Assume that (H) is satisfied. Let $x_{0}\in L^{\infty}(\mcO)$, $X$ be the corresponding solution to \eqref{eq:BTW_SOC} and set
\[
\tau_{0}(\o)=\inf\{t\ge0|X_{t}(\o)=0,\text{ for a.e. }\xi\in\mcO\}.
\]
Then finite time extinction holds, i.e.
\[
\P[\tau_{0}<\infty]=1.
\]
\end{thm}
\begin{rem}
In Theorem \ref{thm:main-intro} we restrict to essentially bounded initial conditions for simplicity. In fact, as we will see in Section \ref{sec:FTE} the extinction time $\tau_{0}$ can be bounded in terms of appropriate $L^{p}(\mcO)$ norms of $x_{0}$ (with $p$ depending on the dimension $d$). Due to continuity of $X_{t}$ in the initial condition this is easily seen to imply finite time extinction for $x_{0}\in L^{p}(\mcO)$ as well. 
\end{rem}

\begin{rem}[Spatially homogeneous noise]
Assume that the functions $f_{k}$ are constant. Then define $F(t)=\int_{0}^{t}e^{\mu_{r}+\td\mu r}dr$, $G(t)=F^{-1}(t).$ An informal computation suggests that $u_{t}:=Y_{G(t)}$ solves
\[
\partial_{t}u\in\D\sgn(u),
\]
i.e. the case of spatially homogeneous noise may entirely be reduced to the deterministic situation, for which finite time extinction has been first proven in \cite{DD79} (cf. also \cite{B13} for a more recent approach). The informal computation introduced above may be made rigorous by first considering non-degenerate, non-singular, smooth approximations (as we will do below, cf. Section \ref{sec:construction} below) for which the transformation follows from the classical chain-rule.

The same remark applies to \eqref{eq:BTW_SOC} if the stochastic part is given in Stratonovich form by choosing $F(t)=\int_{0}^{t}e^{\mu_{r}}dr$. 
\end{rem}

\subsubsection{Outline of the poof\label{sub:proof-outline}}

As layed out above, our analysis is based on the transformed equation \eqref{eq:TPME-intro-1}. The proof consists of two main ingredients:
\begin{enumerate}
\item A uniform control on $\|X_{t}\|_{p}$ for all $p\ge1$.
\item An energy inequality for a weighted $L^{1}$ norm of  $Y_{t}$.
\end{enumerate}
While on an intuitive level the arguments used in this paper to prove finite time extinction become quite clear by considering an approximation of the $\sgn$ function by $r^{[m]}:=|r|^{m-1}r$ ($m\downarrow0$) it is necessary to choose a more complicated, non-singular, non-degenerate approximation in the rigorous proof. Therefore, we start by giving an informal outline of the proof based on $r^{[m]}\to\sgn$. 

\textbf{Step 1:} A uniform control on $\|X_{t}\|_{p}$ for all $p\ge1$

Let $Y_{t}$ be the solution to
\[
\partial Y_{t}\in e^{\mu_{t}}\D Y_{t}^{[m]}-\td\mu Y_{t},
\]
for some $m>0$. Then we may informally compute
\begin{align}
\partial_{t}\int_{\mcO}e^{p\td\mu t}|Y_{t}|^{p}d\xi= & p\int_{\mcO}e^{p\td\mu t}Y_{t}{}^{[p-1]}e^{\mu_{t}}\D Y_{t}{}^{[m]}d\xi\nonumber \\
= & -p\int_{\mcO}e^{\mu_{t}+p\td\mu t}\nabla Y{}^{[p-1]}\nabla Y_{t}^{[m]}d\xi-p\int_{\mcO}Y_{t}^{[p-1]}\nabla e^{\mu_{t}+p\td\mu t}\nabla Y_{t}^{[m]}d\xi\nonumber \\
= & -(p-1)mp\int_{\mcO}e^{\mu_{t}+p\td\mu t}|Y_{t}|^{p-2+m-1}|\nabla Y_{t}|^{2}d\xi\nonumber \\
 & -pm\int_{\mcO}Y_{t}{}^{[p-1+m-1]}\nabla e^{\mu_{t}+p\td\mu t}\nabla Y_{t}d\xi\nonumber \\
= & -\frac{4(p-1)mp}{(p+m-1)^{2}}\int_{\mcO}e^{\mu_{t}+p\td\mu t}\left(\nabla|Y_{t}|^{\frac{p+m-1}{2}}\right)^{2}d\xi\label{eq:ineq-m-approx}\\
 & -\frac{pm}{p+m-1}\int_{\mcO}\nabla|Y_{t}|^{p+m-1}\nabla e^{\mu_{t}+p\td\mu t}d\xi\nonumber \\
= & -\frac{4(p-1)mp}{(p+m-1)^{2}}\int_{\mcO}e^{\mu_{t}+p\td\mu t}\left(\nabla|Y_{t}|^{\frac{p+m-1}{2}}\right)^{2}d\xi\nonumber \\
 & +\frac{pm}{p+m-1}\int_{\mcO}|Y_{t}|^{p+m-1}\D e^{\mu_{t}+p\td\mu t}d\xi,\nonumber 
\end{align}
for all $p\ge1$. Taking $p>1$ and then $m\to0$ we may ``deduce'' from this
\[
\partial_{t}\int_{\mcO}e^{p\td\mu t}|Y_{t}|^{p}d\xi\le0.
\]
Note that for fix $m>0$ this does not follow, since the second term in \eqref{eq:ineq-m-approx} does not vanish. This is the reason why our analysis applies to the BTW model only and not to general fast diffusion equations with $m>0$. In order to turn the above bound on $Y$ into a bound on $X$ we need to control the amount of energy added to the system by the random perturbation. Assuming a mild decay condition on the level sets of $\td\mu$ (cf. Remark \ref{rmk:hyp_noise} below) we obtain that condition (H) is satisfied. This implies
\begin{align}
\int_{\mcO}|X_{t}|^{p}d\xi & =\int_{\mcO}e^{p(-\mu_{t}-\td\mu t)}e^{p\td\mu t}|Y_{t}|^{p}d\xi\nonumber \\
 & \le C_{1}\left(\int_{\mcO}e^{(p+\tau)\td\mu t}|Y_{t}|^{(p+\tau)}d\xi\right)^{\frac{p}{p+\tau}}\label{eq:intro-x-lp-bound}\\
 & \le C_{1}\left(\int_{\mcO}|x_{0}|^{(p+\tau)}d\xi\right)^{\frac{p}{p+\tau}},\nonumber 
\end{align}
i.e.
\[
\|X_{t}\|_{p}\le C_{1}\|x_{0}\|_{p+\tau},
\]
for all $p\ge1$, $\tau>0$, $t\ge t_{0}=t_{0}(p,\tau,\o)$, with $C_{1}=C_{1}(p,\tau,\o)$.
\begin{rem}

\begin{enumerate}
\item While we obtain a uniform bound on each $L^{p}$-norm of $X_{t}$ for large times $t\ge t_{0}$, we do \textit{not} obtain such a uniform bound on the $L^{\infty}$-norm of $X_{t}$ since the geometric Brownian motions $t\mapsto e^{-\mu_{t}(\xi)-\td\mu(\xi)t}$ are \textit{not} necessarily pathwise uniformly bounded in $\xi\in\mcO$. As compared to the deterministic case, this leads to additional difficulties in the proof of finite time extinction.
\item In the derivation of the $L^{p}$ bound of $X_{t}$ presented above, we use that the noise is given in Itô form. It is due to the Itô correction term $\td\mu$ in \eqref{eq:TPME-intro-1} that we may uniformly control $\int_{\mcO}e^{p\td\mu t}|Y_{t}|^{p}d\xi$ (and not only $\int_{\mcO}|Y_{t}|^{p}d\xi$), which in turn is essential in \eqref{eq:intro-x-lp-bound}. \\
In fact, the estimate relies purely on the noise part, since by taking $p>1$ and then $m\to0$ the parts in \eqref{eq:ineq-m-approx} that are due to the diffusive term vanish. 
\end{enumerate}
\end{rem}
\textbf{Step 2:} An energy inequality for a weighted $L^{1}$ norm of  $Y_{t}$.

We now develop the crucial energy estimate to prove finite time extinction. Let $\vp$ be the classical solution to
\begin{align*}
\D\vp & =-1,\quad\text{on }\mcO\\
\vp & =1,\quad\text{on }\partial\mcO.
\end{align*}
Note $1\le\vp\le\|\vp\|_{\infty}=:C_{\vp}$. As in \eqref{eq:ineq-m-approx} we informally compute
\begin{align}
\partial_{t}\int_{\mcO}e^{-\mu_{s}}\vp|Y_{t}|^{p}d\xi= & -\frac{4(p-1)mp}{(p+m-1)^{2}}\int_{\mcO}e^{\mu_{t}-\mu_{s}}\vp\left(\nabla|Y_{t}|^{\frac{p+m-1}{2}}\right)^{2}d\xi\label{eq:first_step_main_est}\\
 & +\frac{pm}{p+m-1}\int_{\mcO}|Y_{t}|^{p+m-1}\D e^{\mu_{t}-\mu_{s}}\vp d\xi.\nonumber 
\end{align}
In order to prove finite time extinction the first term on the right hand side will be crucial and we aim to let $m\to0,p\to1$ simultaneously in such a way that the constant $\frac{4(p-1)mp}{(p+m-1)^{2}}$ does not vanish (in contrast to step one). For example, we may choose $p=m+1$ and obtain
\begin{align*}
\partial_{t}\int_{\mcO}e^{-\mu_{s}}\vp|Y_{t}|^{m+1}d\xi= & -(m+1)\int_{\mcO}e^{\mu_{t}-\mu_{s}}\vp\left(\nabla|Y_{t}|^{m}\right)^{2}d\xi\\
 & +\frac{(m+1)}{2}\int_{\mcO}|Y_{t}|^{2m}\D e^{\mu_{t}-\mu_{s}}\vp d\xi.
\end{align*}
In the limit $m\to0$ we may then expect
\begin{align}
\partial_{t}\int_{\mcO}e^{-\mu_{s}}\vp|Y_{t}|d\xi & =-\int_{\mcO}e^{\mu_{t}-\mu_{s}}\vp\left(\nabla\eta\right)^{2}d\xi+\frac{1}{2}\int_{\mcO}\eta^{2}\D e^{\mu_{t}-\mu_{s}}\vp d\xi,\label{eq:FTE-basic-ineq}
\end{align}
where $\eta$ is a selection from $\sgn(Y)$, i.e. $\eta_{t}(\xi)\in\sgn(Y_{t}(\xi))$ for a.e. $(t,\xi)\in[0,T]\times\mcO$. The crucial point is that if we choose $m\to0,p\to1$ such that the first term in \eqref{eq:first_step_main_est} does not vanish, then also the second one is preserved in the limit. This makes the proof of finite time extinction more intriguing than in the deterministic case where the perturbative second term is not present.

\textbf{Step 3: }Deducing finite time extinction

The principal idea is that 
\[
\D e^{\mu_{t}-\mu_{s}}\vp=e^{\mu_{t}-\mu_{s}}(-1+2\nabla\vp\cdot\nabla(\mu_{t}-\mu_{s})+\vp(|\nabla(\mu_{t}-\mu_{s})|^{2}+\D(\mu_{t}-\mu_{s}))
\]
is non-positive if $\|\mu_{t}-\mu_{s}\|_{C^{2}(\mcO)}$ is sufficiently small and hence we may drop the last term in \eqref{eq:FTE-basic-ineq} on such intervals $[s,t]$. For the sake of simplicity of this introductory overview let us restrict to $d=1$. In higher dimensions $d\ge2$ the control given by the dissipative term $\int_{\mcO}e^{\mu_{t}-\mu_{s}}\vp\left(\nabla\eta\right)^{2}d\xi$ in \eqref{eq:FTE-basic-ineq} is much weaker and the argument leading to finite time extinction is more subtle. For $d=1$ we have $H_{0}^{1}\hookrightarrow L^{\infty}$ . Restricting to intervals $[s,t]$ such that
\begin{equation}
\sup_{r\in[s,t]}\D e^{\mu_{r}-\mu_{s}}\vp\le0\label{eq:interval_ineq}
\end{equation}
we obtain from \eqref{eq:FTE-basic-ineq}: 
\begin{align*}
\int_{\mcO}e^{-\mu_{s}}\vp|Y_{t}|d\xi\le\int_{\mcO}e^{-\mu_{s}}\vp|Y_{s}|d\xi & -\int_{s}^{t}\left(\inf_{\xi\in\mcO}e^{\mu_{r}-\mu_{s}}\right)\|\eta_{r}\|_{\infty}dr.
\end{align*}
By step one we observe 
\begin{align*}
\int_{\mcO}e^{-\mu_{s}}\vp|Y_{s}| & =\int_{\mcO}e^{-\mu_{s}-\td\mu s}\vp e^{\td\mu s}|Y_{s}|d\xi\\
 & \le C_{1}C_{\vp}\|x_{0}\|_{1+\tau},
\end{align*}
for all $\tau>0$, $s\ge t_{0}=t_{0}(\tau,\o)$ and with $C_{1}=C_{1}(\tau,\o)$. Moreover, since $ $$\|\eta_{t}\|_{\infty}=0$ implies $Y_{t}\equiv0$ we may deduce
\begin{equation}
\int_{\mcO}e^{-\mu_{s}}\vp|Y_{t}|d\xi\le C_{1}C_{\vp}\|x_{0}\|_{1+\tau}-\int_{s}^{t}\left(\inf_{\xi\in\mcO}e^{\mu_{r}-\mu_{s}}\right)dr\vee0,\label{eq:d=00003D1-FTE}
\end{equation}
for all intervals $[s,t]$ such that \eqref{eq:interval_ineq} is satisfied and $s\ge t_{0}$. Since 
\[
\|\mu_{t}-\mu_{s}\|_{C^{2}(\mcO)}\le\left(\sum_{k=1}^{N}\|f_{k}\|_{C^{2}(\mcO)}\right)|\b_{t}-\b_{s}|
\]
for \eqref{eq:interval_ineq} to be satisfied we have to restrict to intervals $[s,t]$ where $|\b_{t}-\b_{s}|$ remains small. Due to properties of Brownian motion (cf. Lemma \ref{lem:constant_BM} below) we may find such intervals $[s,t]$ of arbitrary length and hence \eqref{eq:d=00003D1-FTE} implies finite time extinction (with extinction time $\tau_{0}$ depending on $x_{0}$ only via its $L^{1+\tau}$-norm).
\begin{rem}
We note that the methods leading to finite time extinction introduced above do \textit{not}\textbf{ }rely on the presence of noise. In fact, if $\mu\equiv0$, then \eqref{eq:d=00003D1-FTE} reduces to the corresponding estimate from the deterministic case. In particular, no non-degeneracy condition (as e.g. assuming $\td\mu\ge\d>0$ on $\mcO$ as for the result on exponential decay proven in \cite{BR12b}) has to be supposed.
\end{rem}

\subsection{Notation}

In the following let $\mcO\subseteq\R^{d}$ be a bounded, open set with smooth boundary $\partial\mcO.$ For $s\le t,$ $s,t\in\R$ we let $\mcO_{[s,t]}:=[s,t]\times\mcO$ and $\mcO_{T}:=\mcO_{[0,T]}$. For $p\ge1$ we let $L^{p}(\mcO)$ be the usual Lebesgue spaces with norm $\|\cdot\|_{p}:=\|\cdot\|_{L^{p}(\mcO)}$. For $\vp\in L^{\infty}(\mcO)$ we define the weighted Lebesgue space $L_{\vp}^{p}(\mcO)$ to be the space of equivalence classes of measurable functions $f$ such that 
\[
\|f\|_{L_{\vp}^{p}(\mcO)}:=\left(\int_{\mcO}|f(\xi)|^{p}\vp(\xi)d\xi\right)^{\frac{1}{p}}<\infty.
\]
For notational convenience we set $\|\cdot\|_{\vp}:=\|f\|_{L_{\vp}^{1}(\mcO)}$. The spaces $C^{m,n}(\mcO_{T})$ are defined to be spaces of functions on $\mcO_{T}$ with $m$ continuous derivatives in time and $n$ continuous derivatives in space. We let $H_{0}^{1}(\mcO)$ be the first order Sobolev space with zero Dirichlet boundary conditions endowed with the norm
\[
\|f\|_{H_{0}^{1}(\mcO)}:=\int_{\mcO}|\nabla f(\xi)|^{2}d\xi
\]
and let $H^{-1}$ be its dual. We let $\sgn$ denote the maximal monotone extension of the sign function. We write $a\lesssim b$ if there is a constant $C$ such that $a\le Cb$. The constants $C,C_{1},C_{2}$ will denote generic constants that may change value from line to line. For every $p\in[1,\infty]$ we let $p^{*}\in[1,\infty]$ denote the dual exponent, i.e. $\frac{1}{p}+\frac{1}{p^{*}}=1$ (with the convention $\frac{1}{\infty}=0$). We further define $\b=(\b^{k})_{k=1,\dots,N}$ to be an $\R^{N}$-valued standard Brownian motion, without loss of generality given by its canonical realization on $C_{0}(\R_{+};\R^{N})$. We let $(\mcF_{t})_{t\in\R_{+}}$ be the canonical filtration generated by $\b$ with completion $(\bar{\mcF}_{t})_{t\in\R_{+}}$.

\subsection{Overview of the contents}

In Section \ref{sec:existence} we will prove the existence of solutions to \eqref{eq:TPME-intro-1} and some key energy estimates. In the following Section \ref{sec:tranformation} the transformation of \eqref{eq:BTW_SOC} into \eqref{eq:TPME-intro-1} will be justified by proving that for the solution $Y$ to \eqref{eq:TPME-intro-1} constructed in Section \ref{sec:existence} setting $X_{t}:=e^{-\mu_{t}}Y_{t}$ yields a solution to \eqref{eq:BTW_SOC}. The proof of finite time extinction will then be given in Section \ref{sec:FTE}. In the final Section \ref{sec:exp_decay} we prove a pointwise estimate on $X_{t}$ implying exponential convergence to zero on sets $K\subseteq\mcO$ for which $\inf_{\xi\in K}\td\mu(\xi)>0.$

\section{Existence of solutions\label{sec:existence}}

In this section we will construct solutions to the transformed equation \eqref{eq:TPME-intro-1}. In this construction we will work with a fixed realization of the Brownian motion, i.e. we consider $\mu_{t}:=\sum_{k=1}^{N}f_{k}\b_{t}^{k}(\o)$ for an arbitrary, fixed $\o\in\O$. In fact, the precise structure of $\mu$ does not matter for the construction and we consider \eqref{eq:TPME-intro-1} for an arbitrary functions $\mu\in C^{0,2}(\mcO_{T})$ and $\td\mu\in C^{1}(\mcO)$ nonnegative. In particular, we may replace $\b$ by any continuous stochastic process, e.g. fractional Brownian motion. We note, however, that we will use special properties of the Brownian motion in the proof of finite time extinction in Section \ref{sec:FTE} below.

Let us define what we mean by a solution to
\begin{align}
\partial_{t}Y_{t} & \in e^{\mu_{t}}\D\phi(Y_{t})-\td\mu Y_{t},\quad\text{on }\mcO_{T}\label{eq:general_TPME}\\
0 & \in\phi(Y_{t}),\quad\text{on }\partial\mcO,\nonumber 
\end{align}
with $Y_{0}=y_{0}$ and $\phi$ being a possibly multi-valued map. 
\begin{defn}
\label{def:soln}Let $y_{0}\in L^{\infty}(\mcO)$. A tuple $(Y,\eta)$ with $Y\in L^{2}(\mcO_{T})\cap W^{1,2}([0,T];H^{-1})$ and $\eta\in L^{2}([0,T];H_{0}^{1}(\mcO))$ is said to be a solution to \eqref{eq:general_TPME} if
\begin{align*}
\frac{d}{dt}Y_{t} & =e^{\mu_{t}}\D\eta_{t}-\td\mu Y_{t},\ \text{in }H^{-1}\ \text{for a.e. }t\in[0,T]\\
Y_{0} & =y_{0}
\end{align*}
and $\eta_{t}(\xi)\in\phi(Y_{t}(\xi))$ for a.e. $(t,\xi)\in\mcO_{T}$.\end{defn}
\begin{rem}
Let $y_{0}\in L^{\infty}(\mcO)$, $Y\in L^{2}(\mcO_{T})\cap W^{1,2}([0,T];H^{-1})$ and $\eta\in L^{2}([0,T];H_{0}^{1}(\mcO))$. Then $(Y,\eta)$ is a solution to \eqref{eq:general_TPME} in the sense of Definition \ref{def:soln} iff
\[
\int_{\mcO}Y_{t}\vp d\xi=\int_{\mcO}y_{0}\vp d\xi-\int_{0}^{t}\int_{\mcO}\nabla\eta_{r}\cdot\nabla e^{\mu_{r}}\vp d\xi dr-\int_{0}^{t}\int_{\mcO}\td\mu Y_{r}\vp d\xi dr,\quad\text{for a.e. }t\ge0,
\]
for all $\vp\in H_{0}^{1}(\mcO)$ and $\eta_{t}(\xi)\in\phi(Y_{t}(\xi))$ for a.e. $(t,\xi)\in\mcO_{T}$.\end{rem}
\begin{prop}
\label{prop:uniqueness_general_TPME}Suppose $\phi$ is a monotone, Lipschitz continuous function. Let $y_{0}^{(i)}\in L^{\infty}(\mcO)$ and $(Y^{(i)},\eta^{(i)})$ be solutions to \eqref{eq:general_TPME} in the sense of Definition \ref{def:soln}, $i=1,2$. Then there is a $C>0$ such that
\[
\|Y_{t}^{(1)}-Y_{t}^{(2)}\|_{H^{-1}}^{2}\le e^{Ct}\|y_{0}^{(1)}-y_{0}^{(2)}\|_{H^{-1}}^{2},\quad\forall t\in[0,T].
\]
In particular, solutions to \eqref{eq:general_TPME} are unique.\end{prop}
\begin{proof}
By the chain-rule
\begin{align*}
\frac{d}{dt}\|Y_{t}^{(1)}-Y_{t}^{(2)}\|_{H^{-1}}^{2}= & 2(e^{\mu_{t}}\D(\phi(Y_{t}^{(1)})-\phi(Y_{t}^{(2)})),Y_{t}^{(1)}-Y_{t}^{(2)})_{H^{-1}}\\
 & -2(\td\mu(Y_{t}^{(1)}-Y_{t}^{(2)}),Y_{t}^{(1)}-Y_{t}^{(2)})_{H^{-1}}.
\end{align*}
Since $e^{\mu_{t}}\D f=\D e^{\mu_{t}}f-2\nabla e^{\mu_{t}}\cdot\nabla f-e^{\mu_{t}}\D f$ for all $f\in H_{0}^{1}(\mcO)$ we obtain
\begin{align}
 & (e^{\mu_{t}}\D(\phi(Y_{t}^{(1)})-\phi(Y_{t}^{(2)})),Y_{t}^{(1)}-Y_{t}^{(2)})_{H^{-1}}\nonumber \\
 & =-(e^{\mu_{t}}(\phi(Y_{t}^{(1)})-\phi(Y_{t}^{(2)})),Y_{t}^{(1)}-Y_{t}^{(2)})_{2}\label{eq:uniqueness_1}\\
 & -2(\nabla e^{\mu_{t}}\nabla(\phi(Y_{t}^{(1)})-\phi(Y_{t}^{(2)})),Y_{t}^{(1)}-Y_{t}^{(2)})_{H^{-1}}\nonumber \\
 & -((\phi(Y_{t}^{(1)})-\phi(Y_{t}^{(2)}))\D e^{\mu_{t}},Y_{t}^{(1)}-Y_{t}^{(2)})_{H^{-1}}.\nonumber 
\end{align}
Since $\phi$ is Lipschitz and monotone:
\begin{align*}
 & -(e^{\mu_{t}}(\phi(Y_{t}^{(1)})-\phi(Y_{t}^{(2)})),Y_{t}^{(1)}-Y_{t}^{(2)})_{2}\\
 & =-\int_{\mcO}e^{\mu_{t}}(\phi(Y_{t}^{(1)})-\phi(Y_{t}^{(2)}))(Y_{t}^{(1)}-Y_{t}^{(2)})d\xi\\
 & \le-\frac{1}{\|\phi\|_{Lip}}\int_{\mcO}e^{\mu_{t}}|\phi(Y_{t}^{(1)})-\phi(Y_{t}^{(2)})|^{2}d\xi.
\end{align*}
Moreover,
\begin{align*}
 & -2(\nabla e^{\mu_{t}}\nabla(\phi(Y_{t}^{(1)})-\phi(Y_{t}^{(2)})),Y_{t}^{(1)}-Y_{t}^{(2)})_{H^{-1}}\\
 & \le\ve\|\nabla(\phi(Y_{t}^{(1)})-\phi(Y_{t}^{(2)}))\|_{H^{-1}}^{2}+C_{\ve}\|Y_{t}^{(1)}-Y_{t}^{(2)}\|_{H^{-1}}^{2}\\
 & \le\ve\|\phi(Y_{t}^{(1)})-\phi(Y_{t}^{(2)})\|_{2}^{2}+C_{\ve}\|Y_{t}^{(1)}-Y_{t}^{(2)}\|_{H^{-1}}^{2}
\end{align*}
and the third term in \eqref{eq:uniqueness_1} may be estimated similarly. Choosing $\ve>0$ small enough yields
\begin{align*}
 & (e^{\mu_{t}}\D(\phi(Y_{t}^{(1)})-\phi(Y_{t}^{(2)})),Y_{t}^{(1)}-Y_{t}^{(2)})_{H^{-1}}\lesssim\|Y_{t}^{(1)}-Y_{t}^{(2)}\|_{H^{-1}}^{2},
\end{align*}
which implies the claim.
\end{proof}
We aim to construct solutions to
\begin{equation}
\partial_{t}Y_{t}\in e^{\mu_{t}}\D\sgn(Y_{t})-\td\mu Y_{t}\label{eq:TPME}
\end{equation}
via a smooth, non-degenerate, non-singular approximation of the right-hand side. In order to prove convergence of the approximating solutions it is convenient to employ a three step argument. First, we will consider a Lipschitz (non-singular) approximation of the nonlinearity, i.e.
\[
\partial_{t}Y_{t}^{(\ve)}\in e^{\mu_{t}}\D\phi^{(\ve)}(Y_{t}^{(\ve)})-\td\mu Y_{t}^{(\ve)},\quad\ve>0,
\]
where $\phi^{(\ve)}$ is the Yosida approximation of $\sgn$, then a vanishing viscosity (non-degenerate) approximation, i.e.
\begin{equation}
\partial_{t}Y_{t}^{(\ve,\d)}=e^{\mu_{t}}\D\phi^{(\ve)}(Y_{t}^{(\ve,\d)})+\d e^{\mu_{t}}\D Y_{t}^{(\ve,\d)}-\td\mu Y_{t}^{(\ve,\d)},\quad\ve,\d>0\label{eq:2nd_approx}
\end{equation}
and in the last step we consider smooth approximations $\phi^{(\tau,\ve)},\mu^{(\tau)},\td\mu^{(\tau)}$ :
\begin{equation}
\partial_{t}Y_{t}^{(\tau,\ve,\d)}=e^{\mu_{t}^{(\tau)}}\D\phi^{(\tau,\ve)}(Y_{t}^{(\tau,\ve,\d)})+\d e^{\mu_{t}^{(\tau)}}\D Y_{t}^{(\tau,\ve,\d)}-\td\mu^{(\tau)}Y_{t}^{(\tau,\ve,\d)},\label{eq:3rd_approx}
\end{equation}
with $\tau,\ve,\d>0$. The advantage of keeping $\d>0$ in the first step lies in the resulting continuity of $t\mapsto Y_{t}^{(\ve,\d)}$ in $L^{2}(\mcO)$,  which will be needed to obtain the key energy bound proving finite time extinction (cf. Lemma \ref{lem:main_estimate_eps-delta}  below).

In order to justify the limiting procedures $\tau,\ve,\d\to0$ we require uniform a-priori estimates on $Y^{(\tau,\ve,\d)}$ that will be obtained in the following section.

\subsection{Approximate equation, a-priori bounds}

In this section, we consider PDE of the type
\begin{align}
\partial_{t}Y_{t} & =e^{\mu_{t}}\D\phi(Y_{t})+\d e^{\mu_{t}}\D Y_{t}-\td\mu Y_{t},\quad\text{on }\mcO_{T}\label{eq:smooth_TPME}\\
Y_{t} & =0,\quad\text{on }\partial\mcO,\nonumber 
\end{align}
with $Y_{0}=y_{0}$, $\d>0$, $\mu$, $\td\mu$, $y_{0}$ and $\phi$ being smooth functions, $\td\mu\ge0$, $\phi$ monotone and $\phi(0)=0$. Let $\psi:\R\to\R$ be such that $\dot{\psi}=\phi$. Existence of classical solutions to \eqref{eq:smooth_TPME} follows from \cite{LSU67}.
\begin{lem}
\label{lem:lp-bound-1}For all $p\ge1$ and all $t\ge s\ge0$
\begin{align}
\int_{\mcO}e^{p\td\mu t}|Y_{t}|^{p}d\xi\le & \int_{\mcO}e^{p\td\mu s}|Y_{s}|^{p}d\xi+p\int_{s}^{t}\int_{\mcO}\z(Y_{r})\D e^{\mu_{r}+p\td\mu r}d\xi dr\label{eq:smooth_lp_bound}\\
 & +\d\int_{s}^{t}\int_{\mcO}|Y_{r}|^{p}\D e^{\mu_{r}+p\td\mu r}d\xi dr,\nonumber 
\end{align}
where $\z(t):=\int_{0}^{t}r{}^{[p-1]}\dot{\phi}(r)dr$. Moreover, for all $t\ge s\ge0$
\begin{align}
 & \int_{\mcO}e^{2\td\mu t}|Y_{t}|^{2}d\xi+2\d\int_{s}^{t}\int_{\mcO}e^{\mu_{r}+2\td\mu r}|\nabla Y_{r}|^{2}d\xi dr\nonumber \\
 & \le\int_{\mcO}e^{2\td\mu s}|Y_{s}|^{2}d\xi+2\int_{s}^{t}\int_{\mcO}\z(Y_{r})\D e^{\mu_{r}+2\td\mu r}d\xi dr\label{eq:smooth_h01_bound}\\
 & +\d\int_{s}^{t}\int_{\mcO}|Y_{r}|^{2}\D e^{\mu_{r}+2\td\mu r}d\xi dr.\nonumber 
\end{align}
\end{lem}
\begin{proof}
For now let $\psi^{(\ve)}(t):=(|t|^{\frac{p}{2}}+\ve)^{2}$, $\phi^{(\ve)}=\dot{\psi}^{(\ve)}$. We compute
\begin{align}
 & \partial_{t}\int_{\mcO}e^{p\td\mu t}\psi^{(\ve)}(Y_{t})d\xi\nonumber \\
 & =\int_{\mcO}\phi^{(\ve)}(Y_{t})e^{\mu_{t}+p\td\mu t}\D\phi(Y_{t})d\xi+\d\int_{\mcO}\phi^{(\ve)}(Y_{t})e^{\mu_{t}+p\td\mu t}\D Y_{t}d\xi\nonumber \\
 & -\int_{\mcO}e^{p\td\mu t}\td\mu\phi^{(\ve)}(Y_{t})Y_{t}d\xi+p\int_{\mcO}e^{p\td\mu t}\td\mu\psi^{(\ve)}(Y_{t})d\xi\nonumber \\
 & =-\int_{\mcO}\dot{\phi}^{(\ve)}(Y_{t})e^{\mu_{t}+p\td\mu t}\dot{\phi}(Y_{t})|\nabla Y_{t}|^{2}d\xi-\int_{\mcO}\phi^{(\ve)}(Y_{t})\dot{\phi}(Y_{t})\nabla e^{\mu_{t}+p\td\mu t}\nabla Y_{t}d\xi\label{eq:approx_smooth_lp_bound}\\
 & -\d\int_{\mcO}\dot{\phi}^{(\ve)}(Y_{t})e^{\mu_{t}+p\td\mu t}|\nabla Y_{t}|^{2}d\xi-\d\int_{\mcO}\phi^{(\ve)}(Y_{t})\nabla e^{\mu_{t}+p\td\mu t}\cdot\nabla Y_{t}d\xi\nonumber \\
 & -\int_{\mcO}e^{p\td\mu t}\td\mu\phi^{(\ve)}(Y_{t})Y_{t}d\xi+p\int_{\mcO}e^{p\td\mu t}\td\mu\psi^{(\ve)}(Y_{t})d\xi.\nonumber 
\end{align}
Setting $\z^{(\ve)}(t)=\int_{0}^{t}\phi^{(\ve)}(r)\dot{\phi}(r)dr$ we obtain 
\begin{align*}
 & \partial_{t}\int_{\mcO}e^{p\td\mu t}\psi^{(\ve)}(Y_{t})d\xi\\
 & =-\int_{\mcO}\dot{\phi}^{(\ve)}(Y_{t})e^{\mu_{t}+p\td\mu t}\dot{\phi}(Y_{t})|\nabla Y_{t}|^{2}d\xi+\int_{\mcO}\z^{(\ve)}(Y_{t})\D e^{\mu_{t}+p\td\mu t}d\xi\\
 & -\d\int_{\mcO}\dot{\phi}^{(\ve)}(Y_{t})e^{\mu_{t}+p\td\mu t}|\nabla Y_{t}|^{2}d\xi+\d\int_{\mcO}\psi^{\ve}(Y_{t})\D e^{\mu_{t}+p\td\mu t}d\xi\\
 & -\int_{\mcO}e^{p\td\mu t}\td\mu\phi^{(\ve)}(Y_{t})Y_{t}d\xi+p\int_{\mcO}e^{p\td\mu t}\td\mu\psi^{(\ve)}(Y_{t})d\xi.
\end{align*}
In particular,
\begin{align*}
\partial_{t}\int_{\mcO}e^{p\td\mu t}\psi^{(\ve)}(Y_{t})d\xi\le & \int_{\mcO}\z^{(\ve)}(Y_{t})\D e^{\mu_{t}+p\td\mu t}d\xi+\d\int_{\mcO}\psi^{\ve}(Y_{t})\D e^{\mu_{t}+p\td\mu t}d\xi\\
 & -\int_{\mcO}e^{p\td\mu t}\td\mu\phi^{(\ve)}(Y_{t})Y_{t}d\xi+p\int_{\mcO}e^{p\td\mu t}\td\mu\psi^{(\ve)}(Y_{t})d\xi.
\end{align*}
Letting $\ve\to0$ then yields \eqref{eq:smooth_lp_bound}. Arguing as in \eqref{eq:approx_smooth_lp_bound} but with $\psi^{(\ve)}(r)$ replaced by $r^{2}$ and $p=2$ yields \eqref{eq:smooth_h01_bound}.\end{proof}
\begin{lem}
\label{lem:energy_bound}For all $t\ge s\ge0$ and all $\varrho\in C^{2}(\mcO)$
\begin{align*}
\partial_{t}\int_{\mcO}\psi(Y_{t})\varrho d\xi\le & -\int_{\mcO}\varrho e^{\mu_{t}}|\nabla\phi(Y_{t})|^{2}d\xi+\frac{1}{2}\int_{\mcO}\phi(Y_{t})^{2}\D\varrho e^{\mu_{t}}d\xi+\d\int_{\mcO}\psi(Y_{t})\D\varrho e^{\mu_{t}}d\xi.
\end{align*}
\end{lem}
\begin{proof}
We compute
\begin{align*}
 & \partial_{t}\int_{\mcO}\psi(Y_{t})\varrho d\xi\\
 & =\int_{\mcO}\phi(Y_{t})\varrho e^{\mu_{t}}\D\phi(Y_{t})d\xi+\d\int_{\mcO}\phi(Y_{t})\varrho e^{\mu_{t}}\D Y_{t}d\xi-\int_{\mcO}\phi(Y_{t})\varrho\td\mu Y_{t}d\xi\\
 & \le-\int_{\mcO}\varrho e^{\mu_{t}}|\nabla\phi(Y_{t})|^{2}d\xi-\int_{\mcO}\phi(Y_{t})\nabla\varrho e^{\mu_{t}}\nabla\phi(Y_{t})d\xi-\d\int_{\mcO}\phi(Y_{t})\nabla\varrho e^{\mu_{t}}\nabla Y_{t}d\xi\\
 & =-\int_{\mcO}\varrho e^{\mu_{t}}|\nabla\phi(Y_{t})|^{2}d\xi-\frac{1}{2}\int_{\mcO}\nabla\varrho e^{\mu_{t}}\nabla\phi(Y_{t})^{2}d\xi-\d\int_{\mcO}\nabla\psi(Y_{t})\nabla\varrho e^{\mu_{t}}d\xi\\
 & =-\int_{\mcO}\varrho e^{\mu_{t}}|\nabla\phi(Y_{t})|^{2}d\xi+\frac{1}{2}\int_{\mcO}\phi(Y_{t})^{2}\D\varrho e^{\mu_{t}}d\xi+\d\int_{\mcO}\psi(Y_{t})\D\varrho e^{\mu_{t}}d\xi.
\end{align*}

\end{proof}

\subsection{\label{sec:construction}Construction of a solution and energy bounds}

We need to specify the chosen approximation $\phi^{(\tau,\ve)},\phi^{(\ve)}:\R\to\R$ of the sign function in \eqref{eq:2nd_approx}, \eqref{eq:3rd_approx}. Let $\psi(r):=|r|$ and note $\phi:=\sgn=\partial\psi$. We let $J_{\ve}(r):=(1+\ve\sgn)^{-1}$ be the resolvent of $\sgn$ and $\psi^{(\ve)}$ its Moreau-Yosida approximation, i.e.
\[
\psi^{(\ve)}(r):=\inf_{s\in\R}\frac{1}{2\ve}|r-s|^{2}+|s|=\begin{cases}
\frac{r^{2}}{2\ve} & ,\ |r|\le\ve\\
|r|-\frac{\ve}{2} & ,\ |r|>\ve.
\end{cases}
\]
Then $\psi^{(\ve)}\in W^{2,\infty}(\R)$ with 
\[
\phi^{\ve}(r):=\dot{\psi}^{(\ve)}(r)=\begin{cases}
\frac{r}{\ve} & ,\ |r|\le\ve\\
\frac{r}{|r|} & ,\ |r|>\ve.
\end{cases}
\]
We note that $\phi^{(\ve)}$ is the Yosida-approximation of $\phi$, i.e. 
\begin{equation}
\phi^{\ve}(r)=\frac{1}{\ve}(r-J_{\ve}r)\in\phi(J_{\ve}r),\quad\forall r\in\R\label{eq:Yosida}
\end{equation}
and we have 
\[
\dot{\phi}^{(\ve)}=\ddot{\psi}^{(\ve)}(r)=\begin{cases}
\frac{1}{\ve} & ,\ |r|\le\ve\\
0 & ,\ |r|>\ve.
\end{cases}
\]
Moreover, we note
\begin{equation}
|\psi(r)-\psi^{\ve}(r)|=\psi(r)-\psi^{\ve}(r)\le2\ve\label{eq:phi-approx-error}
\end{equation}
and $\phi^{\ve}(r)\le1$. We further let $\mu^{(\tau)}$ and $\td\mu^{(\tau)}\ge0$ be smooth approximations of $\mu,\td\mu$ such that $\|\mu^{(\tau)}-\mu\|_{C^{0,2}(\mcO_{T})},\|\td\mu^{(\tau)}-\td\mu\|_{C^{0}(\mcO)}\le\tau$, $y_{0}^{(\tau)}$ a smooth approximation of $y_{0}$ with $\|y_{0}^{(\tau)}-y_{0}\|_{1}\le\tau$ and $\|y_{0}^{(\tau)}\|_{\infty}\le\|y_{0}\|_{\infty}$, $\psi^{(\tau,\ve)}:=\psi^{(\ve)}\ast\vp^{(\tau)}\in C^{\infty}(\R)$, where $\vp^{(\tau)}$ is a standard Dirac sequence, and consider the three-step approximation
\begin{align}
\partial_{t}Y_{t}^{(\ve)} & \in e^{\mu_{t}}\D\phi^{(\ve)}(Y_{t}^{(\ve)})-\td\mu Y_{t}^{(\ve)},\quad\text{on }\mcO_{T}\label{eq:eps-approx}\\
Y_{0}^{(\ve)} & =y_{0},\quad\text{on }\mcO\nonumber 
\end{align}
then
\begin{align}
\partial_{t}Y_{t}^{(\ve,\d)} & =e^{\mu_{t}}\D\phi^{(\ve)}(Y_{t}^{(\ve,\d)})+\d e^{\mu_{t}}\D Y_{t}^{(\ve,\d)}-\td\mu Y_{t}^{(\ve,\d)},\quad\text{on }\mcO_{T}\label{eq:eps_delta-approx}\\
Y_{0}^{(\ve,\d)} & =y_{0},\quad\text{on }\mcO,\nonumber 
\end{align}
and
\begin{align}
\partial_{t}Y_{t}^{(\tau,\ve,\d)} & =e^{\mu_{t}^{(\tau)}}\D\phi^{(\tau,\ve)}(Y_{t}^{(\tau,\ve,\d)})+\d e^{\mu_{t}^{(\tau)}}\D Y_{t}^{(\tau,\ve,\d)}-\td\mu^{(\tau)}Y_{t}^{(\tau,\ve,\d)},\quad\text{on }\mcO_{T}\label{eq:tau_eps_delta-approx}\\
Y_{0}^{(\ve,\d)} & =y_{0}^{(\tau)},\quad\text{on }\mcO,\nonumber 
\end{align}
with zero Dirichlet boundary conditions. By \cite{LSU67} there is a unique, classical solution $Y^{(\tau,\ve,\d)}$ to \eqref{eq:tau_eps_delta-approx}. We aim to first let $\tau\to0$ then $\d\to0$ and then $\ve\to0$. As outlined above, the advantage of first keeping the approximate viscosity lies in the fact that $t\mapsto Y_{t}^{(\ve,\d)}$ is continuous in $L^{2}(\mcO)$ which will be needed to establish the key energy estimate. 
\begin{rem}
\label{rmk:construction_subseq}In the following we will prove that for all sequences $(\tau_{n,}\ve_{n},\d_{n})\to0$ we may find subsequences $(\tau_{n_{k},}\ve_{n_{l}},\d_{n_{m}})\to0$ such that
\[
Y^{(\tau_{n_{k},}\ve_{n_{l}},\d_{n_{m}})}\xrightarrow{k\to\infty}Y^{(\ve_{n_{l}},\d_{n_{m}})}\xrightarrow{m\to\infty}Y^{(\ve_{n_{l}})}\xrightarrow{l\to\infty}Y
\]
in a weak sense, where $Y$ is a solution to \eqref{eq:TPME}. Since we have uniqueness for \eqref{eq:eps_delta-approx} and \eqref{eq:tau_eps_delta-approx} in fact the whole corresponding sequences converge. In order to prove $\bar{\mcF}_{t}$-adaptedness of $Y$ in Section \ref{sec:tranformation} we will choose a particular sequence $\ve_{n}\to0$ along which the solution $Y$ will be constructed.\end{rem}
\begin{lem}
\label{lem:eps_del_ex}Let $y_{0}\in L^{\infty}(\mcO)$, $\ve,\d>0$. Then there exists a unique solution $Y^{(\ve,\d)}$ to \eqref{eq:eps_delta-approx} in the sense of Definition \ref{def:soln} satisfying $Y^{(\ve,\d)}\in C([0,T];L^{2}(\mcO))$ and
\begin{align}
\int_{\mcO}e^{p\td\mu t}|Y_{t}^{(\ve,\d)}|^{p}d\xi\le & \int_{\mcO}e^{p\td\mu s}|Y_{s}^{(\ve,\d)}|^{p}d\xi+\ve^{p-1}\int_{s}^{t}\int_{\mcO}|\D e^{\mu_{r}+p\td\mu r}|d\xi dr\label{eq:eps_lp_bound}\\
 & +\d\int_{s}^{t}\int_{\mcO}|Y_{r}^{(\ve,\d)}|^{p}\D e^{\mu_{r}+p\td\mu r}d\xi dr,\nonumber 
\end{align}
for all $[s,t]\subseteq\R_{+}$, $p\ge1$. Moreover, for all \textup{$[s,t]\subseteq\R_{+}$} and all nonnegative $\vr\in C^{2}(\mcO)$ we have 
\begin{align}
 & \int_{\mcO}\psi^{(\ve)}(Y_{t}^{(\ve,\d)})\varrho d\xi+\int_{s}^{t}\int_{\mcO}\varrho e^{\mu_{r}}|\nabla\phi^{(\ve)}(Y_{r}^{(\ve,\d)})|^{2}d\xi dr\nonumber \\
 & \le\int_{\mcO}\psi^{(\ve)}(Y_{s}^{(\ve,\d)})\varrho d\xi+\frac{1}{2}\int_{s}^{t}\int_{\mcO}\phi^{(\ve)}(Y_{r}^{(\ve,\d)})^{2}\D\varrho e^{\mu_{r}}d\xi dr\label{eq:eps_main_est}\\
 & +\d\int_{s}^{t}\int_{\mcO}\psi^{(\ve)}(Y_{r}^{(\ve,\d)})\D\varrho e^{\mu_{r}}d\xi dr.\nonumber 
\end{align}
\end{lem}
\begin{proof}
The construction of solutions to \eqref{eq:eps_delta-approx} starts from \eqref{eq:tau_eps_delta-approx}. Since $\ve,\d>0$ are fixed, for simplicity we will suppress them in the notation of $Y^{(\tau,\ve,\d)},Y^{(\ve,\d)}$ in the following.

\textbf{Step 1: }A-priori bounds

From Lemma \ref{lem:lp-bound-1} we have
\begin{align*}
 & \int_{\mcO}e^{2\td\mu^{(\tau)}t}|Y_{t}^{(\tau)}|^{2}d\xi+2\d\int_{s}^{t}\int_{\mcO}e^{\mu_{r}^{(\tau)}+2\td\mu^{(\tau)}r}|\nabla Y_{r}^{(\tau)}|^{2}d\xi dr\\
 & \le\int_{\mcO}e^{2\td\mu^{(\tau)}s}|Y_{s}^{(\tau)}|^{2}d\xi+2\int_{s}^{t}\int_{\mcO}\z^{(\tau,\ve)}(Y_{r}^{(\tau)})\D e^{\mu_{r}^{(\tau)}+2\td\mu^{(\tau)}r}d\xi dr\\
 & +\d\int_{s}^{t}\int_{\mcO}|Y_{r}^{(\tau)}|^{2}\D e^{\mu_{r}^{(\tau)}+2\td\mu^{(\tau)}r}d\xi dr,
\end{align*}
and 
\begin{align}
\int_{\mcO}e^{p\td\mu^{(\tau)}t}|Y_{t}^{(\tau)}|^{p}d\xi\le & \int_{\mcO}e^{p\td\mu^{(\tau)}s}|Y_{s}^{(\tau)}|^{p}d\xi+p\int_{s}^{t}\int_{\mcO}\z^{(\tau,\ve)}(Y_{r}^{(\tau)})\D e^{\mu_{r}^{(\tau)}+p\td\mu^{(\tau)}r}d\xi dr\nonumber \\
 & +\d\int_{s}^{t}\int_{\mcO}|Y_{r}^{(\tau)}|^{p}\D e^{\mu_{r}^{(\tau)}+p\td\mu^{(\tau)}r}d\xi dr,\label{eq:eps_lp}
\end{align}
 for all $p\ge1$, where $\z^{(\tau,\ve)}(t)=\int_{0}^{t}r{}^{[p-1]}\dot{\phi}^{(\tau,\ve)}(r)dr$. For $p\ge1$ we note
\begin{align}
\z^{(\tau,\ve)}(t) & =\int_{0}^{t}r{}^{[p-1]}\dot{\phi}^{(\tau,\ve)}(r)dr\nonumber \\
 & \le\frac{\ve^{p-1}}{p}\left|\frac{t}{\ve}\wedge\frac{\ve+\tau}{\ve}\right|^{p}\label{eq:zeta_bound-1}\\
 & \le\frac{\ve^{p-1}}{p}\left|\frac{\ve+\tau}{\ve}\right|^{p},\quad\forall t\in\R.\nonumber 
\end{align}
Hence,
\[
\z^{(\tau,\ve)}\le C<\infty,
\]
uniformly in $\tau>0$ (small enough). Using Gronwall's inequality this yields
\begin{align*}
\sup_{t\in[0,T]} & \int_{\mcO}e^{p\td\mu^{(\tau)}t}|Y_{t}^{(\tau)}|^{p}d\xi\le C<\infty,
\end{align*}
for all $p\ge1$ and 
\[
\d\int_{0}^{T}\int_{\mcO}e^{\mu_{r}^{(\tau)}+2\td\mu^{(\tau)}r}|\nabla Y_{r}^{(\tau)}|^{2}d\xi dr\le C<\infty,
\]
uniformly in $\tau$ (and in $\ve,\d$) .

\textbf{Step 2: }Extraction and identification of a limit

From step one we conclude that $Y^{(\tau)}$ is uniformly bounded in $L^{\infty}([0,T];L^{p}(\mcO))$ for all $p\ge1$ and in $L^{2}([0,T];H_{0}^{1}(\mcO))$. Hence, $\phi^{(\tau,\ve)}(Y^{(\tau)})$ is uniformly bounded in $L^{2}([0,T];H_{0}^{1}(\mcO))$ and $\frac{d}{dt}Y^{(\tau)}$ is uniformly bounded in $L^{2}([0,T];H^{-1})$. Since $H_{0}^{1}(\mcO)\hookrightarrow L^{2}(\mcO)$ is compact, we may use Aubin-Lions compactness (cf. e.g. \cite[Proposition III.1.3]{S97}) to extract subsequences%
\footnote{More precisely, for each sequence $\tau^{n}\to0$ we may extract a subsequence $\tau_{n_{k}}$ such that the claimed convergences hold (cf. Remark \eqref{rmk:construction_subseq}).%
} satisfying
\begin{align}
Y^{(\tau)} & \rightharpoonup^{*}Y,\quad\text{in }L^{\infty}([0,T];L^{p}(\mcO))\text{ and in }L^{2}([0,T];H_{0}^{1}(\mcO)),\ \forall p\ge1,\label{eq:tau_approx_convergence}\\
Y^{(\tau)} & \to Y,\quad\text{in }L^{2}([0,T];L^{2}(\mcO))\text{ and dt\ensuremath{\otimes}d\ensuremath{\xi}}\text{-a.e.},\ \text{for }\tau\to0.\nonumber 
\end{align}
As a consequence (using $\phi^{(\tau,\ve)}\to\phi^{(\ve)}$ uniformly), we also have
\begin{align*}
Y^{(\tau)} & \to Y,\quad\text{in }L^{p}(\mcO_{T})\text{ for all }p\ge1,\\
Y_{t}^{(\tau)} & \to Y_{t},\quad\text{for a.e. }t\in[0,T]\text{ in }L^{p}(\mcO)\text{ for all }p\ge1,\\
\phi^{(\tau,\ve)}(Y^{(\tau)}) & \rightharpoonup\phi^{(\ve)}(Y),\quad\text{in }L^{2}([0,T];H_{0}^{1}(\mcO)),\ \text{for }\tau\to0.
\end{align*}
We aim to prove that $Y$ is a solution to \eqref{eq:eps_delta-approx}. We start by proving $Y_{t}^{(\tau)}\rightharpoonup Y_{t}$ in $H^{-1}$ for all $t\in[0,T]$. Let
\[
\mcK:=\{(Y^{(\tau)},h)_{H^{-1}}|h\in H^{-1},\ \|h\|_{H^{-1}}\le1,\ \tau>0\}\subseteq C([0,T]).
\]
Boundedness of $Y_{t}^{(\tau)}$ in $L^{2}(\mcO)$ implies that $\mcK$ is bounded in $C([0,T])$. Moreover,
\begin{align*}
(Y_{t+s}^{(\tau)}-Y_{t}^{(\tau)},h)_{H^{-1}} & =\int_{t}^{t+s}(\frac{d}{dr}Y^{(\tau)},h)_{H^{-1}}dr\le C\|h\|_{H^{-1}}s^{\frac{1}{2}}.
\end{align*}
Hence, $\mcK$ is a set of equibounded, equicontinuous functions. Therefore, for every $h\in H^{-1},\ \|h\|_{H^{-1}}\le1$ there is a subsequence such that $(Y^{(\tau)},h)_{H^{-1}}\to g$ in $C([0,T])$. Due to \eqref{eq:tau_approx_convergence} we have $g=(Y^{(\tau)},h)_{H^{-1}}$ which implies $Y_{t}^{(\tau)}\rightharpoonup Y_{t}$ in $H^{-1}$ for all $t\in[0,T]$. 

Since $Y^{(\tau)}$ is a classical solution to \eqref{eq:eps_delta-approx} we have
\begin{align*}
\int_{\mcO}Y_{t}^{(\tau)}\varrho d\xi=\int_{\mcO}Y_{s}^{(\tau)}\varrho d\xi & +\int_{s}^{t}\int_{\mcO}\phi^{(\tau,\ve)}(Y_{r}^{(\tau)})\D e^{\mu_{r}^{(\tau)}}\varrho d\xi dr\\
 & +\d\int_{s}^{t}\int_{\mcO}Y_{r}^{(\tau)}\D e^{\mu_{r}^{(\tau)}}\varrho d\xi dr-\int_{s}^{t}\int_{\mcO}\td\mu^{(\tau)}Y_{r}^{(\tau)}\varrho d\xi dr,
\end{align*}
for all $\varrho\in C_{0}^{2}(\mcO)$, $t\ge s\ge0$. Taking the limit $\tau\to0$ yields
\begin{align*}
\int_{\mcO}Y_{t}\varrho d\xi=\int_{\mcO}Y_{s}\varrho d\xi & +\int_{s}^{t}\int_{\mcO}\phi^{(\ve)}(Y_{r})\D e^{\mu_{r}}\varrho d\xi dr\\
 & +\d\int_{s}^{t}\int_{\mcO}Y_{r}\D e^{\mu_{r}}\varrho d\xi dr-\int_{s}^{t}\int_{\mcO}\td\mu Y_{r}\varrho d\xi dr,
\end{align*}
for all $\varrho\in C_{0}^{2}(\mcO)$, $t\ge s\ge0$. Since $\phi^{(\ve)}(Y),Y\in L^{2}([0,T];H_{0}^{1}(\mcO))$ this is equivalent to 
\[
\frac{d}{dt}Y_{t}=e^{\mu_{t}}\D\phi^{(\ve)}(Y_{r})+\d e^{\mu_{t}}\D Y_{t}-\td\mu Y_{t},\quad\text{in }H^{-1}\ \text{for a.e. }t\in[0,T].
\]
In particular, we have $Y\in W^{1,2}([0,T];H^{-1})$. Since also $Y\in L^{2}([0,T];H_{0}^{1}(\mcO))$, from \cite[Proposition III.1.2]{S97} we obtain $Y\in C([0,T];L^{2}(\mcO))$. Boundedness in $L^{\infty}([0,T];L^{p}(\mcO))$ for each $p\ge1$ then implies $Y\in C([0,T];L^{p}(\mcO))$ for all $p\ge1.$

\textbf{Step 3: }Proof\textbf{ }of \eqref{eq:eps_lp_bound}, \eqref{eq:eps_main_est}

The inequality \eqref{eq:eps_lp_bound} follows from \eqref{eq:eps_lp} and \eqref{eq:zeta_bound-1} by taking $\tau\to0$ and using $Y\in C([0,T];L^{p}(\mcO))$ for all $p\ge1$. Similarly, \eqref{eq:eps_main_est} follows from Lemma \ref{lem:energy_bound} and the locally uniform convergence $\psi^{(\tau,\ve)}\to\psi^{(\ve)}$.\end{proof}
\begin{prop}
\label{prop:eps_ex}Let $y_{0}\in L^{\infty}(\mcO)$, $\ve>0$. Then there exists a unique solution $Y^{(\ve)}$ to \eqref{eq:eps-approx} in the sense of Definition \ref{def:soln} satisfying $Y^{(\ve)}\in C([0,T];H^{-1})$,
\begin{align}
\int_{\mcO}e^{p\td\mu t}|Y_{t}^{(\ve)}|^{p}d\xi & \le\int_{\mcO}|y_{0}|^{p}d\xi+C\ve^{p-1}\int_{0}^{t}\int_{\mcO}\D e^{\mu_{r}+p\td\mu r}d\xi dr,\quad\forall t\in[0,T],\label{eq:lp-bound-eps}
\end{align}
for all  $p\ge1$. Moreover, 
\begin{align}
 & \int_{\mcO}|Y_{t}^{(\ve)}|d\xi+\int_{0}^{t}\int_{\mcO}e^{\mu_{r}}|\nabla\phi^{(\ve)}(Y^{(\ve)})|^{2}d\xi dr\label{eq:del_gradient_bound}\\
 & \le\int_{\mcO}|y_{0}|d\xi+\frac{1}{2}\int_{0}^{t}\int_{\mcO}|\D e^{\mu_{r}}|d\xi+C\ve.\nonumber 
\end{align}
In addition, $t\mapsto Y_{t}^{(\ve)}$ is weakly continuous in $L^{p}(\mcO)$ for all $p\ge1$.\textup{}\end{prop}
\begin{proof}
The construction of solutions to \eqref{eq:eps-approx} starts from \eqref{eq:eps_delta-approx} and Lemma \ref{lem:eps_del_ex}. 

\textbf{Step 1: }A-priori bounds

From \eqref{eq:eps_main_est} (with $\varrho\equiv1$) we have
\begin{align}
 & \int_{\mcO}\psi^{(\ve)}(Y_{t}^{(\ve,\d)})d\xi+\int_{s}^{t}\int_{\mcO}e^{\mu_{r}}|\nabla\phi^{(\ve)}(Y_{r}^{(\ve,\d)})|^{2}d\xi dr\nonumber \\
 & \le\int_{\mcO}\psi^{(\ve)}(Y_{s}^{(\ve,\d)})d\xi+\frac{1}{2}\int_{s}^{t}\int_{\mcO}|\D e^{\mu_{r}}|d\xi dr\label{eq:eps_del_gradient_bound}\\
 & +\d\int_{s}^{t}\int_{\mcO}\psi^{(\ve)}(Y_{r}^{(\ve,\d)})\D e^{\mu_{r}}d\xi dr.\nonumber 
\end{align}

\textbf{Step 2: }Extraction and identification of a limit

Due to \eqref{eq:eps_lp_bound} and \eqref{eq:eps_del_gradient_bound} we may argue as in Lemma \ref{lem:eps_del_ex} to extract subsequences satisfying
\begin{align*}
Y^{(\ve,\d)} & \rightharpoonup^{*}Y^{(\ve)},\quad\text{in }L^{\infty}([0,T];L^{p}(\mcO)),\ \forall p\ge1,\\
Y^{(\ve,\d)} & \to Y^{(\ve)},\quad\text{in }L^{2}([0,T];H^{-1}),\\
\phi^{(\ve)}(Y^{(\ve,\d)}) & \rightharpoonup\eta^{(\ve)},\quad\text{in }L^{2}([0,T];H_{0}^{1}(\mcO)),\ \text{for }\d\to0.
\end{align*}
Note that due to the lack of a uniform $L^{2}([0,T];H_{0}^{1}(\mcO))$ bound on $Y^{(\ve,\d)}$ for $\d\to0$ we may only deduce strong convergence in $L^{2}([0,T];H^{-1})$ as compared to strong convergence in $L^{2}([0,T];L^{2}(\mcO))$ in Lemma \ref{lem:eps_del_ex}. Arguing as in Lemma \ref{lem:eps_del_ex} we further have
\[
Y_{t}^{(\ve,\d)}\rightharpoonup Y_{t}^{(\ve)},\quad\text{in }H^{-1}\text{ for all }t\in[0,T].
\]
We aim to identify $(Y^{(\ve)},\eta^{(\ve)})$ as a solution to \eqref{eq:eps-approx}. As in Lemma \ref{lem:eps_del_ex} we obtain
\begin{align*}
\int_{\mcO}Y_{t}^{(\ve)}\varrho d\xi=\int_{\mcO}Y_{s}^{(\ve)}\varrho d\xi & +\int_{s}^{t}\int_{\mcO}\eta_{r}^{(\ve)}\D e^{\mu_{r}}\varrho d\xi dr-\int_{s}^{t}\int_{\mcO}\td\mu Y_{r}^{(\ve)}\varrho d\xi dr,
\end{align*}
for all $\varrho\in C_{0}^{2}(\mcO)$, $t\ge s\ge0$ and subsequently $Y^{(\ve)}\in W^{1,2}([0,T];H^{-1})$. Continuity of $t\mapsto Y_{t}^{(\ve)}$ in $H^{-1}$ and uniform boundedness in $L^{p}(\mcO)$ then imply weak continuity of $t\mapsto Y_{t}^{(\ve)}$ in $L^{p}(\mcO)$ for all $p\ge1.$

It remains to identify $\eta^{(\ve)}$. For this we consider the convex, lower semicontinuous functional
\begin{align*}
\Psi^{(\ve)}(x) & :=\int_{0}^{T}\int_{\mcO}\psi^{(\ve)}(x_{t}(\xi))d\xi dt,\quad x\in L^{2}([0,T]\times\mcO).
\end{align*}
Then $\partial\Psi^{(\ve)}:L^{2}([0,T]\times\mcO)\to L^{2}([0,T]\times\mcO)$ with 
\begin{align*}
\partial\Psi^{(\ve)}(x) & =\{\eta^{(\ve)}=\phi^{(\ve)}(x)\}
\end{align*}
being a maximal monotone operator. By monotonicity of $\phi^{(\ve)}$ we have
\begin{align*}
\int_{0}^{T}\int_{\mcO}(\phi^{(\ve)}(Y^{(\ve,\d)})-\phi^{(\ve)}(z))(Y^{(\ve,\d)}-z)d\xi dt & \ge0,
\end{align*}
for all $z\in L^{2}([0,T]\times\mcO)$. Taking the limit $\d\to0$ we obtain
\begin{align*}
\int_{0}^{T}\int_{\mcO}(\eta^{(\ve)}-\phi^{(\ve)}(z))(Y^{(\ve)}-z)d\xi dt & \ge0,
\end{align*}
for all $z\in L^{2}([0,T]\times\mcO)$. By maximal monotonicity this gives $\eta^{(\ve)}\in\partial\Psi^{(\ve)}(Y^{(\ve)})$ and thus $\eta^{(\ve)}=\phi^{(\ve)}(Y^{(\ve)})$. In conclusion, $Y^{(\ve)}$ is a solution to \eqref{eq:eps-approx}.

\textbf{Step 3: }Proof of \eqref{eq:lp-bound-eps}, \eqref{eq:del_gradient_bound}

Equation \eqref{eq:lp-bound-eps} follows from \eqref{eq:eps_lp_bound}. From \eqref{eq:eps_del_gradient_bound} and \eqref{eq:phi-approx-error} we have
\begin{align*}
 & \int_{\mcO}\psi(Y_{t}^{(\ve,\d)})d\xi+\int_{0}^{t}\int_{\mcO}e^{\mu_{r}}|\nabla\phi^{(\ve)}(Y_{r}^{(\ve,\d)})|^{2}d\xi dr\\
 & \le\int_{\mcO}\psi(y_{0})d\xi+C\ve+\frac{1}{2}\int_{0}^{t}\int_{\mcO}|\D e^{\mu_{r}}|d\xi dr\\
 & +\d\int_{0}^{t}\int_{\mcO}\psi^{(\ve)}(Y_{r}^{(\ve,\d)})\D e^{\mu_{r}}d\xi dr.
\end{align*}
Integration against a nonnegative testfunction $\eta\in L^{\infty}([0,T])$ with $\|\eta\|_{1}=1$ and taking the limit $\d\to0$ yields
\begin{align*}
 & \int_{0}^{T}\eta_{t}\int_{\mcO}\psi(Y_{t}^{(\ve)})d\xi dt+\int_{0}^{T}\eta_{t}\int_{0}^{t}\int_{\mcO}e^{\mu_{r}}|\nabla\phi^{(\ve)}(Y_{r}^{(\ve)})|^{2}d\xi drdt\\
 & \le\int_{\mcO}\psi(y_{0})d\xi+C\ve+\frac{1}{2}\int_{0}^{T}\eta_{t}\int_{0}^{t}\int_{\mcO}|\D e^{\mu_{r}}|d\xi drdt.
\end{align*}
Since $t\mapsto Y_{t}^{(\ve)}$ is weakly continuous in $L^{p}(\mcO)$ for each $p\ge1$ this implies \eqref{eq:del_gradient_bound}.
\end{proof}

\begin{thm}
\label{thm:exist_transf}Let $y_{0}\in L^{\infty}(\mcO)$. Then there exists a solution $(Y,\eta)$ to \eqref{eq:general_TPME} in the sense of Definition \ref{def:soln} satisfying
\begin{align}
\int_{\mcO}e^{p\td\mu t}|Y_{t}|^{p}d\xi\le\int_{\mcO}|y_{0}|^{p}d\xi.\label{eq:lp-limit}
\end{align}
In addition, $t\mapsto Y_{t}$ is weakly continuous in $L^{p}(\mcO)$ for all $p\ge1$.

The solution $(Y,\eta)$ can be obtained as a strong-weak limit in $L^{2}([0,T];H^{-1})\times L^{2}([0,T];H_{0}^{1}(\mcO))$ of solutions $(Y^{(\ve)},\eta^{(\ve)}=\phi^{(\ve)}(Y^{(\ve)}))$ constructed in Proposition \ref{prop:eps_ex}.\end{thm}
\begin{proof}
Let $(Y^{(\ve)},\eta^{(\ve)})$ be solutions to \eqref{eq:eps-approx} as constructed in Proposition \ref{prop:eps_ex}. By \eqref{eq:lp-bound-eps}, \eqref{eq:del_gradient_bound} and Aubin-Lions compactness we may extract subsequences such that
\begin{align*}
Y^{(\ve)} & \rightharpoonup^{*}Y,\quad\text{in }L^{\infty}([0,T];L^{p}(\mcO)),\ \forall p\ge1,\\
Y^{(\ve)} & \to Y,\quad\text{in }L^{2}([0,T];H^{-1}),\\
\phi^{(\ve)}(Y^{(\ve)}) & \rightharpoonup\eta,\quad\text{in }L^{2}([0,T];H_{0}^{1}(\mcO)),\ \text{for }\d\to0.
\end{align*}
As in Proposition \ref{prop:eps_ex} we have
\[
Y_{t}^{(\ve)}\rightharpoonup Y_{t},\quad\text{in }H^{-1},\text{ for all }t\in[0,T].
\]
We may then argue as in Proposition \ref{prop:eps_ex} to obtain
\[
\frac{d}{dt}Y_{t}=e^{\mu_{t}}\D\eta_{t}-\td\mu Y_{t},\quad\text{in }H^{-1}
\]
for a.e. $t\in[0,T]$ and $Y\in W^{1,2}([0,T];H^{-1})$. In particular, $Y\in C([0,T];H^{-1})$ which implies weak continuity of $t\mapsto Y_{t}$ in $L^{p}(\mcO)$ due to the $L^{\infty}([0,T];L^{p}(\mcO))$ boundedness.

In order to characterize the limit $\eta$ we may argue similar to Proposition \ref{prop:eps_ex}. For this we consider the convex, lower semicontinuous functionals
\begin{align*}
\Psi(x) & :=\int_{0}^{T}\int_{\mcO}\psi(x_{t}(\xi))d\xi dt,\\
\Psi^{(\ve)}(x) & :=\int_{0}^{T}\int_{\mcO}\psi^{(\ve)}(x_{t}(\xi))d\xi dt,\quad x\in L^{2}([0,T]\times\mcO).
\end{align*}
Then $\partial\Psi,\partial\Psi^{(\ve)}:L^{2}([0,T]\times\mcO)\to L^{2}([0,T]\times\mcO)$ with 
\begin{align*}
\partial\Psi(x) & =\{\eta\in L^{2}([0,T]\times\mcO)|\eta_{t}(\xi)\in\phi(x_{t}(\xi)),\ \text{a.e. }(t,\xi)\in[0,T]\times\mcO\}\\
\partial\Psi^{(\ve)}(x) & =\{\eta^{(\ve)}=\phi^{(\ve)}(x)\}
\end{align*}
being maximal monotone operators. Due to \eqref{eq:phi-approx-error} we have
\[
|\Psi^{\ve}(x)-\Psi(x)|\le\int_{\mcO_{T}}|\psi^{(\ve)}(x_{t}(\xi))-\psi(x_{t}(\xi))|d\xi\otimes dt\le|\mcO_{T}|\ve,\quad\forall x\in L^{2}(\mcO_{T}).
\]
Hence, $\Psi^{\ve}\to\Psi$ in Mosco sense, and thus $\partial\Psi^{\ve}\to\partial\Psi$ in strong graph sense (cf. \cite[Theorem 3.66]{A84}), i.e. for all $(\td z,\td\eta)\in\partial\Psi$ there are $(\td z^{\ve},\td\eta^{\ve}=\phi^{(\ve)}(\td z^{\ve}))\in\partial\Psi^{\ve}$ such that $\td z^{\ve}\to\td z$, $\td\eta^{\ve}\to\td\eta$ in $L^{2}([0,T]\times\mcO)$. By monotonicity of $\phi^{(\ve)}$ we have
\begin{align*}
\int_{0}^{T}\int_{\mcO}(\phi^{(\ve)}(Y^{(\ve)})-\phi^{(\ve)}(\td z^{(\ve)}))(Y^{(\ve)}-\td z^{(\ve)})d\xi dt & \ge0.
\end{align*}
Taking the limit $\ve\to0$ we obtain
\begin{align*}
\int_{0}^{T}\int_{\mcO}(\eta-\td\eta)(Y-\td z)d\xi dt & \ge0,
\end{align*}
for all $(\td z,\td\eta)\in\partial\Psi$. By maximal monotonicity this gives $\eta\in\partial\Psi(Y)$ which implies $\eta_{t}(\xi)\in\phi(Y_{t}(\xi))$ for a.e. $(t,\xi)\in\mcO_{T}$. In conclusion, $(Y,\eta)$ is a solution to \eqref{eq:TPME}.

As in the proof of \eqref{eq:del_gradient_bound}, taking $\ve\to0$ in \eqref{eq:lp-bound-eps} yields \eqref{eq:lp-limit} for almost all $t\ge0$. Then using weak lower-semicontinuity of $x\mapsto\int_{\mcO}e^{p\td\mu t}|x|^{p}d\xi$ on $L^{p}(\mcO)$ and weak continuity of $t\mapsto Y_{t}$ in $L^{p}(\mcO)$ we obtain \eqref{eq:lp-limit} for all $t\ge0$. 
\end{proof}

\section{Transformation\label{sec:tranformation}}

In this section we will give a rigorous justification of the transformation $Y_{t}:=e^{\mu_{t}}X_{t}$ leading to the transformed equation \eqref{eq:TPME-intro-1}, i.e. to
\begin{align}
\partial_{t}Y_{t} & \in e^{\mu_{t}}\D\sgn(Y_{t})-\td\mu Y_{t},\quad\text{on }\mcO_{T}\label{eq:TPME-intro-2}\\
0 & \in\sgn(Y_{t}),\quad\text{on }\partial\mcO.\nonumber 
\end{align}
Since we aim to eventually deduce statements for $X$ from $Y$ we only require the ``back-transformation'', i.e. we aim to show that if $Y$ is a solution to \eqref{eq:TPME-intro-2} constructed in Section \ref{sec:existence} along an appropriate sequence $\ve_{n}\to0$ then $X_{t}:=e^{-\mu_{t}}Y_{t}$ is a solution to \eqref{eq:BTW_SOC}, i.e. to
\begin{align}
dX_{t} & \in\D\sgn(X_{t})dt+\sum_{k=1}^{N}f_{k}X_{t}d\b_{t}^{k},\quad\text{on }\mcO_{T}\label{eq:BTW_SOC-1}\\
0 & \in\sgn(X_{t}),\quad\text{on }\partial\mcO.\nonumber 
\end{align}
In the following, let $f=(f_{k})_{k=1,\dots,N}\in C^{2}(\mcO;\R^{N})$ and $\b=(\b^{k})_{k=1,\dots,N}$ be a standard $\R^{N}$-valued Brownian motion. As before we set $\mu_{t}=-\sum_{k=1}^{N}f_{k}\b_{t}^{k}$ and $\td\mu=\frac{1}{2}\sum_{k=1}^{N}f_{k}^{2}$. Let $S=L^{2}(\mcO)$ and consider the Gelfand triple
\[
S\subseteq H^{-1}\subseteq S^{*}.
\]
Multivalued stochastic evolution inclusions of the type \eqref{eq:BTW_SOC-1} have been studied in \cite{GT11}. In order to also cover approximations to \eqref{eq:BTW_SOC-1} we will recall the setting introduced in \cite[Section 7.1]{GT11} for the more general SPDE of the type 
\begin{align}
dX_{t} & \in\D\phi(e^{\mu_{t}}X_{t})dt+\sum_{k=1}^{N}f_{k}X_{t}d\b_{t}^{k},\quad\text{on }\mcO_{T}\label{eq:general-SFDE}\\
0 & \in\phi(e^{\mu_{t}}X_{t}),\quad\text{on }\partial\mcO,\nonumber 
\end{align}
where $\phi=\partial\psi:\R\to2^{\R}$ is the subgradient of an even, convex, continuous function $\psi$ with $\psi(0)=0$, and for all $\eta\in\phi(r)$: 
\begin{equation}
|\eta|\le C(|r|+1),\quad\forall r\in\R.\label{eq:growth_cdt}
\end{equation}
We then define $\vp(t,u):=\int_{\mcO}\psi(e^{\mu_{t}}u)d\xi$ for $u\in S,t\in[0,T]$ and let $A(t):=\partial\vp(t,\cdot):S\to2^{S^{*}}$. We note
\[
A(t,u)=\{v\in S^{*}|v(\xi)\in\phi(e^{\mu_{t}}u(\xi)),\text{ a.e. }\xi\in\mcO\}
\]
and the growth condition \eqref{eq:growth_cdt} implies
\begin{equation}
\|\eta_{t}\|_{S^{*}}\le C(1+\|e^{-\mu_{t}}\|_{\infty}\|u\|_{S}),\quad\forall\eta_{t}\in A(t,u),\ t\in[0,T].\label{eq:A-growth}
\end{equation}
For $v\in A(t,u)$ we have
\[
\ _{S^{*}}\<v,w\>_{S}=\int_{\mcO}v(\xi)w(\xi)d\xi.
\]

\begin{defn}
\label{def:stoch_soln}A continuous $\bar{\mcF}_{t}$-adapted process $X:[0,T]\times\O\to H^{-1}$ is a solution to \eqref{eq:general-SFDE} if $X\in L^{2}([0,T]\times\O;S)$ and there is an $\eta\in L^{2}([0,T]\times\O;S^{*})$ such that $X$ solves the following integral equation (in $S^{*}$)
\begin{equation}
X_{t}=x_{0}-\int_{0}^{t}\eta_{r}dr+\sum_{k=1}^{N}\int_{0}^{t}f_{k}X_{r}d\b_{r}^{k},\label{eq:stoch_soln}
\end{equation}
$\P$-a.s. for each $t\in[0,T]$ and $\eta\in A(X)$, $dt\otimes\P$-almost everywhere.
\end{defn}
Note that since \eqref{eq:stoch_soln} is satisfied in $S^{*}$, implicitly the Riesz map $\iota=(-\D)^{-1}:H^{-1}\to H_{0}^{1}(\mcO)$ is applied to $X$. Hence, \eqref{eq:stoch_soln} reads
\[
(-\D)^{-1}X_{t}=(-\D)^{-1}x_{0}-\int_{0}^{t}\eta_{r}dr+(-\D)^{-1}\sum_{k=1}^{N}\int_{0}^{t}f_{k}X_{r}d\b_{r}^{k},
\]
again as an equation in $S^{*}$. As applied to $-\D\varrho\in S$ this yields
\[
\int_{\mcO}X_{t}\varrho d\xi=\int_{\mcO}x_{0}\varrho d\xi+\int_{0}^{t}\int_{\mcO}\eta_{r}\D\varrho dr+\sum_{k=1}^{N}\int_{0}^{t}\int_{\mcO}f_{k}X_{r}\varrho d\xi d\b_{r}^{k},
\]
for all $\varrho\in H^{2}(\mcO)\cap H_{0}^{1}(\mcO)$ and a.e. $t\in[0,T]$. Hence, we obtain
\begin{rem}
\label{rmk:weak_form_of_dual_form} Let $X,\eta$ as in Definition \ref{def:stoch_soln}. Then $X$ is a solution to \eqref{eq:general-SFDE} iff
\[
(X_{t},\varrho)_{2}=(x_{0},\varrho)_{2}+\int_{0}^{t}(\eta_{r},\D\varrho)_{2}dr+\sum_{k=1}^{N}\int_{0}^{t}(f_{k}X_{r},\varrho)_{2}d\b_{r}^{k},\quad\text{for a.e. }t\in[0,T]
\]
$\P$-a.s., for all $\varrho\in H^{2}(\mcO)\cap H_{0}^{1}(\mcO)$. 
\end{rem}
Using Itô's formula for Gelfand triples (cf. \cite[Theorem 4.2.5]{PR07}) and monotonicity of the operator $\D\phi:S\to2^{S^{*}}$ yields
\begin{lem}
Solutions to \eqref{eq:general-SFDE} are unique.
\end{lem}
From \cite[Example 7.3]{GT11} we have
\begin{prop}
Let $x_{0}\in L^{2}(\mcO)$. Then there is a unique solution $X$ to \eqref{eq:BTW_SOC-1} in the sense of Definition \ref{def:stoch_soln}. 
\end{prop}
We now proceed to the justification of the transformation $X_{t}:=e^{-\mu_{t}}Y_{t}$. In different contexts analogous transformations have been used e.g. in \cite{BDPR09-2,BR13,G13-2}. In the present situation the proof is more involved, since no uniqueness result for \eqref{eq:TPME-intro-2} is known. 
\begin{thm}
\label{thm:transformation}Let $x_{0}\in L^{\infty}(\mcO)$. Then there is an $\bar{\mcF}_{t}$-adapted solution $Y$ to \eqref{eq:TPME-intro-2} constructed as in Section \ref{sec:existence}. Moreover, $X_{t}:=e^{-\mu_{t}}Y_{t}$ is the unique solution to \eqref{eq:BTW_SOC-1} in the sense of Definition \ref{def:stoch_soln}.
\end{thm}
The proof of Theorem \ref{thm:transformation} proceeds in several steps. The main difficulty is the proof of $\bar{\mcF}_{t}$-adaptedness of $Y$, which does not simply follow from the approximation via $Y^{(\ve_{n})}$ due to the lack of a uniqueness result for \eqref{eq:TPME-intro-2}. Note that the subsequence $\ve^{n_{m}}$ along which $Y^{(\ve_{n_{m}})}$ converges to $Y$ may depend on $\o\in\O$ and thus adaptedness of $Y$ does not (yet) follow from the adaptedness of $Y^{(\ve_{n_{m}})}$. The main idea in this section is to prove convergence (not only along some subsequence) of $Y^{(\ve_{n})}$ by proving convergence on the level of the ``back-transformation'' $X^{(\ve_{n})}:=e^{-\mu}Y^{(\ve_{n})}$. 
\begin{prop}
\label{prop:transformation}Let $x_{0}\in L^{\infty}(\mcO)$ and for all $\o\in\O$ let $(Y(\o),\eta(\o))$ be a solution to
\begin{align*}
\partial_{t}Y_{t} & \in e^{\mu_{t}(\o)}\D\phi(Y_{t})-\td\mu Y_{t},\quad\text{on }\mcO_{T}\\
0 & \in\phi(Y_{t}),\quad\text{on }\partial\mcO,
\end{align*}
with $Y_{0}=x_{0}$ in the sense of Definition \ref{def:soln}. Assume that $Y\in L^{2+\tau}([0,T]\times\O\times\mcO)$ for some $\tau>0$ and that $t\mapsto Y_{t}$ is $\bar{\mcF}_{t}$-adapted in $H^{-1}$. Then $(X:=e^{-\mu}Y,\eta)$ is a solution to \eqref{eq:general-SFDE}.\end{prop}
\begin{proof}
Since $Y\in L^{2+\tau}([0,T]\times\O\times\mcO)$ for some $\tau>0$, using Hölder's inequality and Fernique's Theorem we obtain $X=e^{-\mu}Y\in L^{2}([0,T]\times\O;S)$. Since $\eta(\o)\in\phi(Y(\o))$, $dt\otimes d\xi$-a.e. for all $\o\in\O$ and due to \eqref{eq:growth_cdt} we have $|\eta(\o)|\lesssim1+|Y(\o)|$ for all $\o\in\O$ and thus $\eta\in L^{2}([0,T]\times\O;S^{*})$. From the assumptions it immediately follows that $X_{t}$ is a continuous, $\bar{\mcF}_{t}$-adapted process in $H^{-1}$.

Let $e_{j}\in H_{0}^{1}(\mcO)\cap H^{2}(\mcO)$ be an orthonormal basis of eigenvectors of $-\D$ on $L^{2}(\mcO)$. For all $\xi\in\mcO$ the process $t\mapsto e^{-\mu_{t}(\xi)}$ is a continuous semimartingale satisfying
\begin{align*}
e{}^{-\mu_{t}(\xi)} & =e^{-\mu_{0}(\xi)}+\sum_{k=1}^{N}\int_{0}^{t}f_{k}(\xi)e^{-\mu_{r}(\xi)}d\b_{r}^{k}+\frac{1}{2}\sum_{k=1}^{N}\int_{0}^{t}f_{k}^{2}(\xi)e^{-\mu_{r}(\xi)}dr\\
 & =1+\sum_{k=1}^{N}\int_{0}^{t}f_{k}(\xi)e^{-\mu_{r}(\xi)}d\b_{r}^{k}+\td\mu(\xi)\int_{0}^{t}e^{-\mu_{r}(\xi)}dr.
\end{align*}
By the stochastic Fubini Theorem we have
\[
(e_{j},e{}^{-\mu_{t}}\varrho)_{2}=(e_{j},\varrho)_{2}-\sum_{k=1}^{N}\int_{0}^{t}(e_{j},f_{k}e^{-\mu_{r}}\varrho)_{2}d\b_{r}^{k}+\int_{0}^{t}(e_{j},\varrho\td\mu e^{-\mu_{r}})_{2}dr,
\]
for all $\varrho\in C_{0}^{2}(\mcO)$, $j\in\N$. Since $Y\in W^{1,2}([0,T];H^{-1})$ and $t\mapsto Y_{t}$ is $\bar{\mcF}_{t}$-adapted, $Y$ is an $H^{-1}$-valued semimartingale and hence $t\mapsto\ _{H^{-1}}\<Y_{t},e_{j}\>_{H_{0}^{1}(\mcO)}$ is a semimartingale satisfying
\[
\ _{H^{-1}}\<Y_{t},e_{j}\>_{H_{0}^{1}(\mcO)}=(e_{j},Y_{0})_{2}+\int_{0}^{t}(\D(e_{j}e^{\mu_{r}}),\eta_{r})_{2}dr-\int_{0}^{t}(e_{j},\td\mu Y_{r})_{2}dr,
\]
where $\eta\in\phi(Y)$, $dt\otimes d\xi\otimes\P$-almost everywhere. Now we apply Itô's product rule to obtain
\begin{align*}
 & (e_{j},e{}^{-\mu_{t}}\varrho)_{2}\ _{H^{-1}}\<Y_{t},e_{j}\>_{H_{0}^{1}(\mcO)}\\
 & =(e_{j},\varrho)_{2}(e_{j},y_{0})_{2}+\int_{0}^{t}(e_{j},e{}^{-\mu_{r}}\varrho)_{2}d\ _{H^{-1}}\<Y_{t},e_{j}\>_{H_{0}^{1}(\mcO)}+\int_{0}^{t}(e_{j},Y_{r})_{2}d(e_{j},e{}^{-\mu_{r}}\varrho)_{2}\\
 & =(e_{j},\varrho)_{2}(e_{j},y_{0})_{2}+\int_{0}^{t}(e_{j},e{}^{-\mu_{r}}\varrho)_{2}(\D(e_{j}e^{\mu_{r}}),\eta_{r})_{2}dr-\int_{0}^{t}(e_{j},e{}^{-\mu_{r}}\varrho)_{2}(e_{j},\td\mu Y_{r})_{2}dr\\
 & +\sum_{k=1}^{N}\int_{0}^{t}(e_{j},Y_{r})_{2}(e_{j},f_{k}e^{-\mu_{r}}\varrho)_{2}d\b_{r}^{k}+\int_{0}^{t}(e_{j},Y_{r})_{2}(e_{j},\varrho\td\mu e^{-\mu_{r}})_{2}dr.
\end{align*}
Since
\[
(\D(e_{j}e^{\mu_{r}}),\eta_{r})_{2}=\ _{H^{-1}}\<e^{\mu_{r}}\D\eta_{r},e_{j}\>_{H_{0}^{1}}
\]
summing over $j$ yields (the third and last term on the right hand side cancel)
\begin{align*}
\ _{H^{-1}}\<Y_{t},e{}^{-\mu_{t}}\varrho\>_{H_{0}^{1}(\mcO)}=(\varrho,y_{0})_{2} & +\int_{0}^{t}\ _{H^{-1}}\<e^{\mu_{r}}\D\eta_{r},e{}^{-\mu_{r}}\varrho\>_{H_{0}^{1}}dr\\
 & +\sum_{k=1}^{N}\int_{0}^{t}(Y_{r},f_{k}e^{-\mu_{r}}\varrho)_{2}d\b_{r}^{k}
\end{align*}
and thus
\[
\ _{H^{-1}}\<X_{t},\varrho\>_{H_{0}^{1}(\mcO)}=(\varrho,x_{0})_{2}+\int_{0}^{t}(\D\varrho,\eta_{r})_{2}dr+\sum_{k=1}^{N}\int_{0}^{t}(X_{r},f_{k}\varrho)_{2}d\b_{r}^{k},
\]
for all $\varrho\in C_{0}^{2}(\mcO)$ and all $t\in[0,T]$. By Remark \ref{rmk:weak_form_of_dual_form} this implies the claim.
\end{proof}
In order to apply Proposition \ref{prop:transformation} we need to prove that $Y$ constructed in Theorem \ref{thm:exist_transf} may be chosen to be $\bar{\mcF}_{t}$-adapted in $H^{-1}.$ As outlined above, for this we will prove convergence on the level of the ``back-transformations'' $X^{(\ve)}:=e^{-\mu}Y^{(\ve)}$.
\begin{lem}
\label{lem:vanishing_visc_convergence-1-1}Let $x\in L^{2}(\mcO)$. For all $\ve>0$ let $(X^{(\ve)},\eta^{(\ve)})$ be a solution to \eqref{eq:general-SFDE} with $\phi\equiv\phi^{(\ve)}$ and $\phi^{(\ve)}$ as in Section \ref{sec:construction}. Assume $\sup_{\ve\ge0}\|X^{(\ve)}\|_{L^{2}([0,T]\times\O;S)}\le C$. Then
\[
X^{(\ve)}\to X\quad\text{for }\ve\to0\text{ in }L^{2}(\O;C([0,T];H^{-1})),
\]
where $X$ is a solution to \eqref{eq:general-SFDE} with $\phi=\partial\sgn$. \end{lem}
\begin{proof}
For $\ve_{1},\ve_{2}>0$ we consider two solutions $(X^{(\ve_{i})},\eta^{(\ve_{i})})$, $i=1,2$ to \eqref{eq:general-SFDE}. For notational simplicity let $B(v)(u):=\sum_{k=1}^{N}f_{k}vu^{k}$ for $v\in H^{-1},u\in\R^{N}$ in the following. We further let $L_{2}=L_{2}(\R^{N};H^{-1})$ be the space of Hilbert-Schmidt operators from $\R^{N}$ to $H^{-1}$. By Itô's formula we observe
\begin{align*}
 & \|X_{t}^{(\ve_{1})}-X_{t}^{(\ve_{2})}\|_{H^{-1}}^{2}\\
 & =2\int_{0}^{t}\ _{S^{*}}\<\eta^{(\ve_{1})}-\eta^{(\ve_{2})},X_{r}^{(\ve_{1})}-X_{r}^{(\ve_{2})}\>_{S}dr+\int_{0}^{t}\|B(X_{r}^{(\ve_{1})})-B(X_{r}^{(\ve_{2})})\|_{L_{2}}^{2}dr\\
 & +2\int_{0}^{t}(B(X_{r}^{(\ve_{1})})-B(X_{r}^{(\ve_{2})}),X_{r}^{(\ve_{1})}-X_{r}^{(\ve_{2})})_{H^{-1}}dW_{r}\\
 & =-2\int_{0}^{t}\int_{\mcO}(\phi^{(\ve_{1})}(e^{\mu_{r}}X_{r}^{(\ve_{1})})-\phi^{(\ve_{2})}(e^{\mu_{r}}X_{r}^{(\ve_{2})}))(X_{r}^{(\ve_{1})}-X_{r}^{(\ve_{2})})d\xi dr\\
 & +\int_{0}^{t}\|B(X_{r}^{(\ve_{1})})-B(X_{r}^{(\ve_{2})})\|_{L_{2}}^{2}dr\\
 & +2\int_{0}^{t}(B(X_{r}^{(\ve_{1})})-B(X_{r}^{(\ve_{2})}),X_{r}^{(\ve_{1})}-X_{r}^{(\ve_{2})})_{H^{-1}}dW_{r}.
\end{align*}
Due to \eqref{eq:Yosida} we note
\begin{align*}
(\phi^{(\ve_{1})}(a)-\phi^{(\ve_{2})}(b))\cdot(a-b) & =(\phi^{(\ve_{1})}(a)-\phi^{(\ve_{2})}(b))\cdot(J_{\ve}a-J_{\ve}b)\\
 & +(\phi^{(\ve_{1})}(a)-\phi^{(\ve_{2})}(b))\cdot(a-J_{\ve}a-(b-J_{\ve}b))\\
 & \ge(\phi^{(\ve_{1})}(a)-\phi^{(\ve_{2})}(b))\cdot(\ve_{1}\phi^{(\ve_{1})}(a)-\ve_{2}\phi^{(\ve_{2})}(b))\\
 & \ge-2(\ve_{1}+\ve_{2})\quad\forall a,b\in\R.
\end{align*}
Using the Burkholder-Davis-Gundy inequality we obtain
\begin{align*}
\E\sup_{t\in[0,T]}e^{-Kt}\|X_{t}^{(\ve_{1})}-X_{t}^{(\ve_{2})}\|_{H^{-1}}^{2} & \le C(\ve_{1}+\ve_{2})T,
\end{align*}
for $K>0$ sufficiently large. We conclude
\[
X^{(\ve)}\to X\quad\text{in }L^{2}(\O;C([0,T];H^{-1})).
\]
It remains to identify $X$ as a solution to \eqref{eq:general-SFDE}. Boundedness of $X^{(\ve)}$ in $L^{2}([0,T]\times\O;S)$ and \eqref{eq:A-growth} imply boundedness of $\eta^{(\ve)}$ in $L^{2}([0,T]\times\O;S^{*})$. Hence, we may extract subsequences such that
\begin{align*}
X^{(\ve)} & \rightharpoonup X\quad\text{in }L^{2}([0,T]\times\O;S)\\
\eta^{(\ve)}=\phi^{(\ve)}(e^{\mu}X^{(\ve)}) & \rightharpoonup\eta\quad\text{in }L^{2}([0,T]\times\O;S^{*}).
\end{align*}
Since 
\[
X_{t}^{(\ve)}=x_{0}+\int_{0}^{t}\eta_{r}^{(\ve)}dr+\sum_{k=1}^{N}\int_{0}^{t}f_{k}X_{r}^{(\ve)}d\b_{r}^{k},
\]
as an equation in $S^{*}$, taking the limit $\ve\to0$ we obtain 
\[
X_{t}=x_{0}+\int_{0}^{t}\eta_{r}dr+\sum_{k=1}^{N}\int_{0}^{t}f_{k}X_{r}d\b_{r}^{k}
\]
and it remains to identify $\eta$. Similar to Section \ref{sec:existence} we now consider the convex, lower semicontinuous functionals
\begin{align*}
\Psi(x) & :=\E\int_{0}^{T}\int_{\mcO}\psi(e^{\mu_{t}}x_{t}(\xi))d\xi dt,\\
\Psi^{(\ve)}(x) & :=\E\int_{0}^{T}\int_{\mcO}\psi^{(\ve)}(e^{\mu_{t}}x_{t}(\xi))d\xi dt,\quad x\in L^{2}([0,T]\times\O\times\mcO).
\end{align*}
Then $\partial\Psi,\partial\Psi^{(\ve)}:L^{2}([0,T]\times\O\times\mcO)\to L^{2}([0,T]\times\O\times\mcO)$ are maximal monotone operators. Since $\psi^{(\ve)}\to\psi$ uniformly, we have $\Psi^{(\ve)}\to\Psi$ in Mosco sense, and thus $\partial\Psi^{(\ve)}\to\partial\Psi$ in strong graph sense (cf. \cite[Theorem 3.66]{A84}), i.e. for all $(\td z,\td\eta)\in\partial\Psi$ there are $(\td z^{(\ve)},\td\eta^{(\ve)}=\phi^{(\ve)}(e^{\mu_{t}}\td z^{(\ve)}))\in\partial\Psi^{(\ve)}$ such that $\td z^{(\ve)}\to\td z$, $\td\eta^{(\ve)}\to\td\eta$ in $L^{2}([0,T]\times\O\times\mcO)$. By monotonicity of $\phi^{(\ve)}$ we have
\begin{align*}
\E\int_{0}^{T}\int_{\mcO}(\phi^{(\ve)}(e^{\mu_{r}}X_{r}^{(\ve)})-\phi^{(\ve)}(e^{\mu_{r}}\td z_{r}^{(\ve)})(X_{r}^{(\ve)}-\td z_{r}^{(\ve)})d\xi dr & \ge0.
\end{align*}
Taking the limit $\ve\to0$ we obtain
\begin{align*}
\E\int_{0}^{T}\int_{\mcO}(\eta_{r}-\td\eta_{r})(X_{r}-\td z_{r})d\xi dr & \ge0,
\end{align*}
for all $(\td z,\td\eta)\in\partial\Psi$. By maximal monotonicity this gives $\eta\in\partial\Psi(X)$ which implies $\eta_{t}(\xi)\in\phi(e^{-\mu_{t}(\xi)}X_{t}(\xi))$ for a.e. $(t,\xi)\in\mcO_{T}$. In conclusion, $(X,\eta)$ is a solution to \eqref{eq:general-SFDE} with $\phi=\partial\psi$.
\end{proof}

\begin{proof}[Proof of Theorem \ref{thm:transformation}:]
 In order to prove $\bar{\mcF}_{t}$-adaptedness of $Y^{(\ve,\d)}$ we note that in the construction we may choose the approximation $\mu^{(\tau)}$ of $\mu$ in an $\bar{\mcF}_{t}$-adapted way (e.g. by first shifting the time variable by $\tau$, then mollifying with a standard Dirac sequence). Continuity of the solution to \eqref{eq:tau_eps_delta-approx} with respect to $\mu^{(\tau)}$ is classical. This implies $\bar{\mcF}_{t}$-adaptedness of $Y^{(\tau,\ve,\d)}$. From Lemma \ref{lem:eps_del_ex} and by uniqueness of $Y^{(\ve,\d)}$ we have weak convergence along the full sequence, i.e.
\begin{align*}
Y^{(\tau,\ve,\d)} & \rightharpoonup Y^{(\ve,\d)},\quad\text{in }L^{2}([0,T];H^{-1}).
\end{align*}
Hence, $Y^{(\ve,\d)}$ is $\bar{\mcF}_{t}$-adapted. Analogous reasoning yields $\bar{\mcF}_{t}$-adaptedness of $Y^{(\ve)}$. 

From \eqref{eq:lp-bound-eps} we have
\begin{align*}
\sup_{t\in[0,T]}\E\int_{\mcO}|Y_{t}^{(\ve)}|^{p}d\xi & \le\int_{\mcO}|y_{0}|^{p}d\xi+C\ve^{p-1}\E\int_{0}^{T}\int_{\mcO}|\D e^{\mu_{r}+p\td\mu r}|d\xi dr
\end{align*}
and the right hand side is finite due to Fernique's Theorem. Hence, $Y^{(\ve)}$ is uniformly bounded in $L^{p}([0,T]\times\O\times\mcO)$ for all $p\ge1$. Proposition \ref{prop:transformation} implies that $X^{(\ve)}:=e^{-\mu}Y^{(\ve)}$ is a solution to \eqref{eq:general-SFDE} with $\phi=\phi^{(\ve)}.$ Again employing Fernique's Theorem we note that $X^{(\ve)}$ is uniformly bounded in $L^{2}([0,T]\times\O;S)$. Since $\psi^{(\ve)}$ is the Moreau-Yosida approximation of $\psi$, Lemma \ref{lem:vanishing_visc_convergence-1-1} implies $X^{(\ve)}\to X$ in $L^{2}([0,T]\times\O;H^{-1}).$ Thus, there is a sequence $\ve_{n}\to0$ such that
\[
X^{(\ve_{n})}\to X\quad\P\text{-a.s. in }L^{2}([0,T];H^{-1}).
\]
Since also $Y^{(\ve_{n_{l}})}\rightharpoonup Y$ in $L^{2}([0,T];H^{-1})$ along some subsequence $n_{l}$, we obtain
\begin{align*}
e^{\mu}X=Y & ,\quad\P\text{-a.s.},
\end{align*}
which implies that $Y$ is $\bar{\mcF}_{t}$-adapted, if constructed along this specific sequence $\ve_{n}$. 

Proposition \ref{prop:transformation} then finishes the proof.
\end{proof}

\section{Finite time extinction\label{sec:FTE}}

In this section we will prove finite time extinction via energy methods as outlined in the introduction. We will first prove an energy inequality for the approximating solutions $Y^{(\ve,\d)}$ constructed in Lemma  \ref{lem:eps_del_ex}. By a limiting argument this will imply finite time extinction for $Y$. 

In order to control the amount of energy added to the system by the random perturbation we need to require

\begin{hypothesis}\label{hyp:noise}

Assume that $f=(f_{k})_{k=1,\dots,N}\in C^{2}(\mcO;\R^{N})$ is such that for all $p\ge1$ and a.a. $\o\in\O$ there is a $t_{0}=t_{0}(p,\o)$ such that 
\[
\int_{\mcO}e^{-p\sum_{k=1}^{N}f_{k}(\xi)\b_{t}^{k}(\o)-\frac{1}{2}f_{k}^{2}(\xi)t}d\xi=\int_{\mcO}e^{-p\mu_{t}(\xi,\o)-p\td\mu(\xi)t}d\xi\le C(p,\o)<\infty,
\]
for all $t\ge t_{0}.$

\end{hypothesis}

Based on the law of iterated logarithm, it is not difficult to see that as long as $\td\mu$ is strictly positive Hypothesis \ref{hyp:noise} is satisfied. Trivially, Hypothesis \ref{hyp:noise} is satisfied when no noise is present (i.e. $\mu,\td\mu\equiv0$). More generally, a mild decay condition on the size of the level sets of $\td\mu$ is sufficient to guarantee Hypothesis \ref{hyp:noise}:
\begin{rem}
\label{rmk:hyp_noise}Assume that for each $p\ge1$ there are $\ve_{0},c>0$ such that
\[
\left|\left\{ \xi\in\mcO\Big|0<\sum_{k=1}^{N}f_{k}^{2}(\xi)\le\ve\right\} \right|=|\{0<\td\mu\le\ve\}|\le c|\log(\ve)|^{-p},\quad\forall\ve\le\ve_{0}.
\]
Then Hypothesis \ref{hyp:noise} is satisfied. In particular, Hypothesis \ref{hyp:noise} is satisfied whenever the mass of the sublevel sets of $\td\mu$ decays polynomially, i.e. if $|\{0<\td\mu\le\ve\}|\lesssim\ve^{\d}$ for all $\ve\le\ve_{0}$ and some $\d>0$.\end{rem}
\begin{proof}
For each $\tau>1$, by the law of iterated logarithm we have
\[
|\b_{t}(\o)|\le\tau\sqrt{2t\log_{2}(t)},\quad\text{for all }t\ge t_{0}(\tau,\o),
\]
for $\P$-a.a. $\o\in\O$. For $p\ge1$, we estimate
\begin{align*}
\int_{\mcO}e^{-p\mu_{t}(\xi)-p\td\mu(\xi)t}d\xi\le & \int_{\mcO}e^{p|f||\b_{t}|-p\td\mu t}d\xi\\
\le & \int_{\mcO}e^{p|f|(|\b_{t}|-\frac{|f|}{2}t)}d\xi\\
= & \int_{\mcO}e^{p|f|\sqrt{t}\left(\tau\sqrt{2\log_{2}(t)}-\frac{|f|}{2}\sqrt{t}\right)}d\xi\\
= & \int_{\{|f|\le\frac{2\tau\sqrt{2\log_{2}(t)}}{\sqrt{t}}\}}e^{p|f|\sqrt{t}\left(\tau\sqrt{2\log_{2}(t)}-\frac{|f|}{2}\sqrt{t}\right)}d\xi\\
 & +\int_{\{|f|>\frac{2\tau\sqrt{2\log_{2}(t)}}{\sqrt{t}}\}}e^{p|f|\sqrt{t}\left(\tau\sqrt{2\log_{2}(t)}-\frac{|f|}{2}\sqrt{t}\right)}d\xi\\
\le & \int_{\{|f|\le\frac{2\tau\sqrt{2\log_{2}(t)}}{\sqrt{t}}\}}e^{p|f|\sqrt{t}\tau\sqrt{2\log_{2}(t)}}d\xi+|\mcO|\\
\le & e^{4p\tau^{2}\log_{2}(t)}\left|\left\{ |f|\le\frac{2\tau\sqrt{2\log_{2}(t)}}{\sqrt{t}}\right\} \right|+|\mcO|,
\end{align*}
for all $t\ge t_{0}(\tau,\o)$. Since $\frac{2\tau\sqrt{2\log_{2}(t)}}{\sqrt{t}}\to0$ for $t\to\infty$ we can use the assumption to conclude 
\begin{align*}
\int_{\mcO}e^{-\mu_{t}(\xi)-\td\mu(\xi)t}d\xi & \lesssim|\log(t)|^{4p\tau^{2}}|\log(t^{-\frac{1}{4}})|^{-(4p+1)}+|\mcO|\\
 & \lesssim|\log(t)|^{4p\tau^{2}-(4p+1)}+|\mcO|,
\end{align*}
for all $t\ge t_{0}(\tau,p,\o)$. Choosing $\tau>0$ small enough implies the claim. 
\end{proof}
We will now proceed to prove the key energy estimate in an approximate form for $Y^{(\ve,\d)}$. In the following let $\vp\in C^{2}(\mcO)$ be the classical solution to 
\begin{align*}
\D\vp & =-1,\quad\text{on }\mcO\\
\vp & =1,\quad\text{on }\partial\mcO.
\end{align*}
By the maximum principle we have $1\le\vp\le\|\vp\|_{\infty}=:C_{\vp}$. 
\begin{lem}
\label{lem:main_estimate_eps-delta}Assume that Hypothesis \ref{hyp:noise} is satisfied. Let $y_{0}\in L^{\infty}(\mcO)$, $\ve,\d>0$ and $Y^{(\ve,\d)}$ be the associated solution to \eqref{eq:eps_delta-approx}. Let $\tau>0$ and $[s,t]\subseteq\R_{+}$ such that 
\[
\sup_{r\in[s,t]}\D\vp e^{\mu_{r}-\mu_{s}}\le0
\]
and $s\ge t_{0}=t_{0}(p,\tau,\o)$, with $t_{0}$ as in Hypothesis \ref{hyp:noise}. Then, for all $p>\frac{d}{2}\vee1$ we have
\begin{align*}
 & \int_{\mcO}\psi^{(\ve)}(Y_{t}^{(\ve,\d)})\vp e^{-\mu_{s}}d\xi\\
 & \le\left(\left(h_{2}(s,1,\tau,\d,\ve,\|x_{0}\|_{1+\tau}^{1+\tau})-K^{\ve}\right)^{(1-\a)}-(1-\a)\int_{s}^{t}g^{(\ve,\d)}(r)dr\vee0\right)^{\frac{1}{1-\a}}+K^{\ve},
\end{align*}
where
\begin{align*}
h_{2}(r,p,\tau,\d,\ve,x):= & C_{1}C_{\vp}^{p}\Big(e^{\d\int_{0}^{r}h_{1}(s,p+\tau)ds}x\\
 & +C\ve^{p+\tau-1}\int_{0}^{r}e^{\d\int_{s}^{t}h_{1}(w,p+\tau)dw}h_{1}(s,p+\tau)ds\Big)^{\frac{p}{p+\tau}}\\
h_{1}(r,p):= & \sup_{\xi\in\mcO}|\D e^{\mu_{r}+p\td\mu r}|\\
K^{(\ve)}:= & C_{\vp}\|e^{-\mu_{s}}\|_{\infty}\frac{\ve}{2}
\end{align*}
and $C_{1}=C_{1}(p,\tau,\o)$, $\a=\frac{2p^{*}}{q}<1$,
\[
g^{(\ve,\d)}(r)=\begin{cases}
\left(\inf_{\xi\in\mcO}e^{\mu_{r}-\mu_{s}}\right)\|\phi^{(\ve)}(Y_{r}^{(\ve,\d)})\|_{\infty} & ,\text{ for }d=1\\
\frac{\inf_{\xi\in\mcO}e^{\mu_{r}-\mu_{s}}}{h_{2}(r,p,\tau,\d,\ve,\|x_{0}\|_{p+\tau}^{p+\tau})^{\frac{2p^{*}}{pq}}} & ,\text{ for }d\ge2,
\end{cases}
\]
and
\begin{align}
q & =\infty\text{ if }d=1,\nonumber \\
q & \in(2,\infty)\text{ arbitrary if }d=2,\label{eq:Sobolev constant-1}\\
q & =\frac{2d}{d-2}\text{ if }d\ge3.\nonumber 
\end{align}
\end{lem}
\begin{proof}
From \eqref{eq:eps_main_est} with $\varrho=\vp e^{-\mu_{s}}$ we have
\begin{align*}
 & \int_{\mcO}\psi^{(\ve)}(Y_{v}^{(\ve,\d)})\vp e^{-\mu_{s}}d\xi+\int_{u}^{v}\int_{\mcO}\vp e^{\mu_{r}-\mu_{s}}|\nabla\phi^{(\ve)}(Y_{r}^{(\ve,\d)})|^{2}d\xi dr\\
 & \le\int_{\mcO}\psi^{(\ve)}(Y_{u}^{(\ve,\d)})\vp e^{-\mu_{s}}d\xi,
\end{align*}
for all $[u,v]\subseteq[s,t]$ (where $[s,t]$ is as in the statement). We now use the Sobolev embedding $H_{0}^{1}(\mcO)\hookrightarrow L^{q}(\mcO)$ with $q>2$ as in \eqref{eq:Sobolev constant-1}. The simpler case $d=1$ has already been outlined in the introduction. Hence, we shall restrict to the case $d\ge2$ in the following. We obtain
\begin{align*}
 & \int_{\mcO}\psi^{(\ve)}(Y_{v}^{(\ve,\d)})\vp e^{-\mu_{s}}d\xi+\int_{u}^{v}\inf_{\xi\in\mcO}e^{\mu_{r}-\mu_{s}}\left(\int_{\mcO}|\phi^{(\ve)}(Y_{r}^{(\ve,\d)})|^{q}d\xi\right)^{\frac{2}{q}}dr\\
 & \le\int_{\mcO}\psi^{(\ve)}(Y_{u}^{(\ve,\d)})\vp e^{-\mu_{s}}d\xi.
\end{align*}
We have
\begin{align*}
\int_{\mcO}|\phi^{(\ve)}(Y_{r}^{(\ve,\d)})|^{q}d\xi & =\int_{|Y_{r}^{(\ve,\d)}|\le\ve}|\phi^{(\ve)}(Y_{r}^{(\ve,\d)})|^{q}d\xi+\int_{|Y_{r}^{(\ve,\d)}|>\ve}|\phi^{(\ve)}(Y_{r}^{(\ve,\d)})|^{q}d\xi\\
 & \ge|\{|Y_{r}^{(\ve,\d)}|>\ve\}|.
\end{align*}
This implies
\begin{align}
 & \int_{\mcO}\psi^{(\ve)}(Y_{v}^{(\ve,\d)})\vp e^{-\mu_{s}}d\xi+\int_{u}^{v}\inf_{\xi\in\mcO}e^{\mu_{r}-\mu_{s}}|\{|Y_{r}^{(\ve,\d)}|>\ve\}|^{\frac{2}{q}}dr\label{eq:main-est-2}\\
 & \le\int_{\mcO}\psi^{(\ve)}(Y_{u}^{(\ve,\d)})\vp e^{-\mu_{s}}d\xi.\nonumber 
\end{align}
We note
\begin{align}
 & \int_{\mcO}\psi^{(\ve)}(Y_{v}^{(\ve,\d)})\vp e^{-\mu_{s}}d\xi\nonumber \\
 & =\int_{|Y_{v}^{(\ve,\d)}|\le\ve}\psi^{(\ve)}(Y_{v}^{(\ve,\d)})\vp e^{-\mu_{s}}d\xi+\int_{|Y_{v}^{(\ve,\d)}|>\ve}\psi^{(\ve)}(Y_{v}^{(\ve,\d)})\vp e^{-\mu_{s}}d\xi\nonumber \\
 & \le C_{\vp}\|e^{-\mu_{s}}\|_{\infty}\frac{\ve}{2}+\int_{|Y_{v}^{(\ve,\d)}|>\ve}|Y_{v}^{(\ve,\d)}|\vp e^{-\mu_{s}}d\xi\label{eq:bound_by_mass}\\
 & \le C_{\vp}\|e^{-\mu_{s}}\|_{\infty}\frac{\ve}{2}+\|e^{-\mu_{s}}\vp Y_{v}^{(\ve,\d)}\|_{p}|\{|Y_{v}^{(\ve,\d)}|>\ve\}|^{\frac{1}{p^{*}}},\quad\forall p>1.\nonumber 
\end{align}
From \eqref{eq:eps_lp_bound} we obtain
\begin{align*}
\int_{\mcO}e^{p\td\mu t}|Y_{t}^{(\ve,\d)}|^{p}d\xi\le & e^{\d\int_{0}^{t}h_{1}(r,p)dr}\int_{\mcO}|y_{0}|^{p}d\xi+C\ve^{p-1}\int_{0}^{t}e^{\d\int_{r}^{t}h_{1}(\tau,p)d\tau}h_{1}(r,p)dr,
\end{align*}
with $h_{1}(r,p):=\sup_{\xi\in\mcO}|\D e^{\mu_{r}+p\td\mu r}|$. Thus, by Hypothesis \ref{hyp:noise}: 
\begin{align}
\int_{\mcO}|e^{-\mu_{s}}\vp Y_{r}^{(\ve,\d)}|^{p}= & \int_{\mcO}e^{p(-\mu_{s}-\td\mu s)}e^{-p\td\mu(r-s)}\vp^{p}e^{p\td\mu r}|Y_{r}^{(\ve,\d)}|^{p}d\xi\nonumber \\
\le & C_{1}C_{\vp}^{p}\left(\int_{\mcO}e^{(p+\tau)\td\mu r}|Y_{r}^{(\ve,\d)}|^{p+\tau}d\xi\right)^{\frac{p}{p+\tau}}\nonumber \\
\le & C_{1}C_{\vp}^{p}\Big(e^{\d\int_{0}^{r}h_{1}(s,p+\tau)ds}\int_{\mcO}|x_{0}|^{p+\tau}d\xi\label{eq:Y-bound-1}\\
 & +C\ve^{p+\tau-1}\int_{0}^{r}e^{\d\int_{s}^{t}h_{1}(w,p+\tau)dw}h_{1}(s,p+\tau)ds\Big)^{\frac{p}{p+\tau}}\nonumber \\
=: & h_{2}(r,p,\tau,\d,\ve,\|x_{0}\|_{p+\tau}^{p+\tau}),\nonumber 
\end{align}
for all $p\ge1$ and $r\ge s\vee t_{0}$, where  $C_{1}=C_{1}(p,\tau,\o)$, $t_{0}=t_{0}(p,\tau,\o)$ is as in Hypothesis \ref{hyp:noise}. From \eqref{eq:bound_by_mass} and \eqref{eq:Y-bound-1} we conclude
\begin{align*}
 & |\{|Y_{v}^{(\ve,\d)}|>\ve\}|\\
 & \ge\frac{1}{\|e^{-\mu_{s}}\vp Y_{v}^{(\ve,\d)}\|_{p}^{p^{*}}}\left(\int_{\mcO}\psi^{(\ve)}(Y_{v}^{(\ve,\d)})\vp e^{-\mu_{s}}d\xi-C_{\vp}\|e^{-\mu_{s}}\|_{\infty}\frac{\ve}{2}\right)^{p^{*}}\\
 & \ge\frac{1}{h_{2}(v,p,\tau,\d,\ve,\|x_{0}\|_{p+\tau}^{p+\tau})^{\frac{p^{*}}{p}}}\left(\int_{\mcO}\psi^{(\ve)}(Y_{v}^{(\ve,\d)})\vp e^{-\mu_{s}}d\xi-C_{\vp}\|e^{-\mu_{s}}\|_{\infty}\frac{\ve}{2}\right)^{p^{*}}.
\end{align*}
Using this in \eqref{eq:main-est-2} yields
\begin{align*}
 & \int_{\mcO}\psi^{(\ve)}(Y_{v}^{(\ve,\d)})\vp e^{-\mu_{s}}d\xi+\int_{u}^{v}g^{(\ve,\d)}(r)\left(\int_{\mcO}\psi^{(\ve)}(Y_{r}^{(\ve,\d)})\vp e^{-\mu_{s}}d\xi-K^{(\ve)}\right)^{\frac{2p^{*}}{q}}dr\\
 & \le\int_{\mcO}\psi^{(\ve)}(Y_{u}^{(\ve,\d)})\vp e^{-\mu_{s}}d\xi,
\end{align*}
 for all $[u,v]\subseteq[s,t]$ with 
\begin{align*}
g^{(\ve,\d)}(r): & =\frac{\inf_{\xi\in\mcO}e^{\mu_{r}-\mu_{s}}}{h_{2}(r,p,\tau,\d,\ve,\|x_{0}\|_{p+\tau}^{p+\tau})^{\frac{2p^{*}}{pq}}}\\
K^{(\ve)} & :=C_{\vp}\|e^{-\mu_{s}}\|_{\infty}\frac{\ve}{2}.
\end{align*}
We will require $\a:=\frac{2p^{*}}{q}<1$ which is equivalent to choosing $p>\frac{q}{q-2}$. Note that $\frac{q}{q-2}=\frac{d}{2}$ for $d\ge3$ and $\frac{q}{q-2}$ can be chosen arbitrarily close to $1$ for $d=2$. Hence, let $p>\frac{d}{2}$ arbitrary, fixed.

With $f(u):=\int_{\mcO}\psi^{(\ve)}(Y_{u}^{(\ve,\d)})\vp e^{-\mu_{s}}d\xi=:\|Y_{u}^{(\ve,\d)}\|_{\vp e^{-\mu_{s}}}$ we obtain
\begin{align*}
f(v)+\int_{u}^{v}g^{(\ve,\d)}(r)\left(f(r)-K^{(\ve)}\right)dr\le & f(u),\quad\text{for all }[u,v]\subseteq[s,t].
\end{align*}
Note that $f$ is continuous since $Y^{(\ve,\d)}\in C([0,T];L^{2}(\mcO))$. Hence, from Lemma \ref{lem:ODE} we obtain
\[
\|Y_{t}^{(\ve,\d)}\|_{\vp e^{-\mu_{s}}}\le\left(\left(\|Y_{s}^{(\d)}\|_{\vp e^{-\mu_{s}}}-K^{\ve}\right)^{(1-\a)}-(1-\a)\int_{s}^{t}g^{(\ve,\d)}(r)dr\vee0\right)^{\frac{1}{1-\a}}+K^{\ve}.
\]
From \eqref{eq:Y-bound-1} (with $p=1$ we conclude)
\begin{align*}
 & \|Y_{t}^{(\ve,\d)}\|_{\vp e^{-\mu_{s}}}\\
 & \le\left(\left(h_{2}(s,1,\tau,\d,\ve,\|x_{0}\|_{1+\tau}^{1+\tau})-K^{\ve}\right)^{(1-\a)}-(1-\a)\int_{s}^{t}g^{(\ve,\d)}(r)dr\vee0\right)^{\frac{1}{1-\a}}+K^{\ve}.
\end{align*}

\end{proof}

We may now derive the key energy estimate for $Y$ from Lemma \ref{lem:main_estimate_eps-delta} by taking the limits $\d\to0$ then $\ve\to0$. We obtain
\begin{lem}
\label{lem:delta_main_bound-1-1}Assume that Hypothesis \ref{hyp:noise} is satisfied. Let $y_{0}\in L^{\infty}(\mcO)$ and $Y$ be a solution to \eqref{eq:TPME} constructed in Theorem \ref{thm:exist_transf}. Let $\tau>0$ and $[s,t]\subseteq\R_{+}$ such that 
\[
\sup_{r\in[s,t]}\D\vp e^{\mu_{r}-\mu_{s}}\le0
\]
and $s\ge t_{0}=t_{0}(p,\tau,\o)$, with $t_{0}$ as in Hypothesis \ref{hyp:noise}. Then, for all $p>\frac{d}{2}\vee1$ we have
\[
\int_{\mcO}\psi(Y_{t})\vp e^{-\mu_{s}}d\xi\le\left(\left(C_{1}C_{\vp}\|x_{0}\|_{1+\tau}\right)^{(1-\a)}-(1-\a)\int_{s}^{t}g(r)dr\vee0\right)^{\frac{1}{1-\a}},
\]
where $C_{1}=C_{1}(p,\tau,\o)$, $\a=\frac{2p^{*}}{q}<1$,
\[
g(r)=\begin{cases}
\left(\inf_{\xi\in\mcO}e^{\mu_{r}-\mu_{s}}\right)\|\eta_{r}\|_{\infty} & ,\text{ for }d=1\\
\frac{\inf_{\xi\in\mcO}e^{\mu_{r}-\mu_{s}}}{(C_{1}C_{\vp}\|x_{0}\|_{p+\tau})^{\frac{2p^{*}}{q}}} & ,\text{ for }d\ge2,
\end{cases}
\]
and
\begin{align*}
q & =\infty\text{ if }d=1,\\
q & \in(2,\infty)\text{ arbitrary if }d=2,\\
q & =\frac{2d}{d-2}\text{ if }d\ge3.
\end{align*}
\end{lem}
\begin{proof}
We first recall that the weak convergences $Y^{(\ve,\d)}\rightharpoonup Y^{(\ve)}$ and $Y^{(\ve)}\rightharpoonup Y$ hold as weak limits in $L^{p}(\mcO_{T})$ for all $p>1$. Integrating the main estimate from Lemma \ref{lem:main_estimate_eps-delta} against a nonnegative testfunction $\eta\in C^{1}([s,t])$ with $\|\eta\|_{L^{1}([s,t])}=1$ and using weak lower semicontinuity of $v\mapsto\int_{s}^{t}\eta_{r}\int_{\mcO}\psi^{(\ve)}(v_{r})\vp e^{-\mu_{s}}d\xi dr$ on $L^{p}(\mcO_{T})$ we obtain by taking $\d\to0$:
\begin{align*}
 & \int_{s}^{t}\eta_{r}\int_{\mcO}\psi^{(\ve)}(Y_{r}^{(\ve)})\vp e^{-\mu_{s}}d\xi dr\\
 & \le\int_{s}^{t}\eta_{r}\left(\left(h_{2}(s,1,\tau,\ve,\|x_{0}\|_{1+\tau}^{1+\tau})-K^{\ve}\right)^{(1-\a)}-(1-\a)\int_{s}^{r}g^{(\ve)}(\tau)d\tau\vee0\right)^{\frac{1}{1-\a}}dr+K^{\ve}.
\end{align*}
Since $t\mapsto Y_{t}^{(\ve)}$ is weakly continuous on $L^{p}(\mcO)$ this implies 
\begin{align*}
 & \int_{\mcO}\psi^{(\ve)}(Y_{r}^{(\ve)})\vp e^{-\mu_{s}}d\xi\\
 & \le\left(\left(h_{2}(s,1,\tau,\ve,\|x_{0}\|_{1+\tau}^{1+\tau})-K^{\ve}\right)^{(1-\a)}-(1-\a)\int_{s}^{r}g^{(\ve)}(\tau)d\tau\vee0\right)^{\frac{1}{1-\a}}+K^{\ve},
\end{align*}
for all $r\in[s,t]$. Due to \eqref{eq:phi-approx-error} this implies
\begin{align*}
 & \int_{\mcO}\psi(Y_{r}^{(\ve)})\vp e^{-\mu_{s}}d\xi\\
 & \le\left(\left(h_{2}(s,1,\tau,\ve,\|x_{0}\|_{1+\tau}^{1+\tau})-K^{\ve}\right)^{(1-\a)}-(1-\a)\int_{s}^{r}g^{(\ve)}(\tau)d\tau\vee0\right)^{\frac{1}{1-\a}}+K^{\ve}+C\ve,
\end{align*}
for all $r\in[s,t]$. Using the same reasoning as for $\d\to0$ allows to take the limit $\ve\to0$, which yields
\[
\int_{\mcO}\psi(Y_{t})\vp e^{-\mu_{s}}d\xi\le\left(\left(h_{2}(1,\tau,\|x_{0}\|_{1+\tau}^{1+\tau})\right)^{(1-\a)}-(1-\a)\int_{s}^{t}g(\tau)d\tau\vee0\right)^{\frac{1}{1-\a}},
\]
where
\begin{align*}
h_{2}(p,\tau,x)=\lim_{\ve\to0}\lim_{\d\to0} & h_{2}(s,p,\tau,\d,\ve,x):=C_{1}C_{\vp}^{p}x^{\frac{p}{p+\tau}}
\end{align*}
and
\[
g(r)=\lim_{\ve\to0}\lim_{\d\to0}g^{(\ve,\d)}(r)=\begin{cases}
\left(\inf_{\xi\in\mcO}e^{\mu_{r}-\mu_{s}}\right)\|\eta_{r}\|_{\infty} & ,\text{ for }d=1\\
\frac{\inf_{\xi\in\mcO}e^{\mu_{r}-\mu_{s}}}{C_{1}C_{\vp}^{p}\left(\int_{\mcO}|x_{0}|^{p+\tau}d\xi\right)^{\frac{2p^{*}}{q(p+\tau)}}} & ,\text{ for }d\ge2.
\end{cases}
\]
$ $\end{proof}
\begin{thm}
\label{thm:FTE-1}Assume that Hypothesis \ref{hyp:noise} is satisfied%
\footnote{In fact, we only need Hypothesis \ref{hyp:noise} to hold for certain $p\in\R_{+}$ depending on the integrability of the initial condition $x_{0}$ and the dimension $d$. %
}. Let $x_{0}\in L^{\infty}(\mcO)$, $X$ be the unique solution to \eqref{eq:BTW_SOC} and let 
\[
\tau_{0}(\o):=\inf\{t\ge0|X_{t}(\o)=0,\text{ for a.e. }\xi\in\mcO\}.
\]
Then finite time extinction holds, i.e. 
\[
\P[\tau_{0}<\infty]=1.
\]
The extinction time $\tau_{0}(\o)$ may be chosen uniformly for $x_{0}$ bounded in $L^{p}(\mcO)$, for any
\[
p>\begin{cases}
1 & ,\text{ if }d=1\\
\frac{d}{2} & ,\text{ if }d\ge2.
\end{cases}
\]
\end{thm}
\begin{proof}
Let $Y$ be a solution to \eqref{eq:TPME-intro-2} as constructed in Theorem \ref{thm:transformation}. Let $p>\frac{d}{2}$ ($p>1$ in case $d=1$) and set $\td p:=\frac{p+\frac{d}{2}}{2}$, $\tau:=p-\td p$. From Lemma \ref{lem:delta_main_bound-1-1} we recall: For all intervals $[s,t]$ with $s\ge t_{0}=t_{0}(\td p,\tau,\o)=t_{0}(p,d,\o)$ such that $\sup_{r\in[s,t]}\D\vp e^{\mu_{r}-\mu_{s}}\le0$ we have 
\begin{align}
\int_{\mcO}|Y_{t}|\vp e^{-\mu_{s}}d\xi & \le\left(\left(C_{1}C_{\vp}\|x_{0}\|_{1+\tau}\right)^{1-\a}-(1-\a)\int_{s}^{t}g(r)dr\vee0\right)^{\frac{1}{1-\a}}\label{eq:limit_main_est}\\
 & \le\left(\left(C_{1}C_{\vp}\|x_{0}\|_{p}\right)^{1-\a}-(1-\a)\int_{s}^{t}g(r)dr\vee0\right)^{\frac{1}{1-\a}}\nonumber 
\end{align}
with
\[
g(r)=\begin{cases}
\left(\inf_{\xi\in\mcO}e^{\mu_{r}-\mu_{s}}\right)\|\eta_{r}\|_{\infty} & ,\text{ for }d=1\\
\frac{\inf_{\xi\in\mcO}e^{\mu_{r}-\mu_{s}}}{(C_{1}C_{\vp}\|x_{0}\|_{p})^{\frac{2\td p^{*}}{q}}} & ,\text{ for }d\ge2,
\end{cases}
\]
and the same constants as in Lemma \ref{lem:delta_main_bound-1-1}. Note that $\td p>\frac{d}{2}$ and thus $\a=\frac{2\td p^{*}}{q}<1$. We will now restrict to the more difficult case $d\ge2$, while $d=1$ follows similarly.

By Lemma \ref{lem:constant_BM} there is a set $\O_{0}\subset\O$ of full $\P$-measure, such that for all $\ve>0,\o\in\O$ we may find arbitrarily large intervals $[s,t]$ with $s\ge t_{0}$ such that $\sup_{r\in[s,t]}|\b_{r}(\o)-\b_{s}(\o)|\le\ve$ for all $\o\in\O_{0}$. Also note
\begin{align*}
\D\vp e^{\mu_{r}-\mu_{s}} & =e^{\mu_{r}-\mu_{s}}\left(-1+2\nabla(\mu_{r}-\mu_{s})\cdot\nabla\vp+\vp(|\nabla(\mu_{r}-\mu_{s})|^{2}+\D(\mu_{r}-\mu_{s}))\right)\\
 & \to-1,
\end{align*}
for $\|\mu_{\cdot}-\mu_{s}\|_{C^{0,2}([s,t]\times\mcO)}\to0$. Thus, for $\ve$ small enough we have
\[
\sup_{r\in[s,t]}\D\vp e^{\mu_{r}-\mu_{s}}\le0,
\]
whenever $\sup_{r\in[s,t]}|\b_{r}(\o)-\b_{s}(\o)|\le\ve$. In conclusion, we may choose arbitrarily large intervals $[s,t]$ such that $\sup_{r\in[s,t]}\D\vp e^{\mu_{r}-\mu_{s}}\le0$ and 
\[
g(r)\ge\left(\frac{1}{2(1\vee C_{1}C_{\vp}\|x_{0}\|_{p})}\right)^{\frac{2\td p^{*}}{q}}.
\]
On such intervals we have from \eqref{eq:limit_main_est}:
\begin{align*}
 & \int_{\mcO}|Y_{t}|\vp e^{-\mu_{s}}d\xi\\
 & \le\left((C_{1}C_{\vp}\|x_{0}\|_{p})^{1-\a}-|t-s|(1-\a)\left(\frac{1}{2(1\vee C_{1}C_{\vp}\|x_{0}\|_{p})}\right)^{\frac{2\td p^{*}}{q}}\vee0\right)^{\frac{1}{1-\a}}.
\end{align*}
Since we may choose $|t-s|$ arbitrary large this implies that for all $\o\in\O_{0}$ there is a $\tau_{0}(\o)$ such that $ $
\[
Y_{\tau_{0}}=0,\text{ a.e. in }\mcO.
\]
The claim now follows from Theorem \ref{thm:transformation}.
\end{proof}

\section{Decay due to (Itô-)noise\label{sec:exp_decay}}

Using similar ideas as in the proof of Theorem \ref{thm:FTE-1} we may (partially) sharpen a result obtained in \cite[Theorem 2.3]{BR12b}. More precisely, assuming $\td\mu=\frac{1}{2}\sum_{k=1}^{N}f_{k}^{2}>0$ it has been shown in \cite{BR12b} that
\[
\int_{K}X_{t}d\xi\le\|x_{0}\|_{2}|K|^{\frac{1}{2}}e^{\sup_{K}\td\mu^{\frac{1}{2}}\left(\sum_{k=1}^{N}(\b_{t}^{k})^{2}\right)^{\frac{1}{2}}}e^{-\frac{t}{2}\inf_{K'}\td\mu},
\]
for every compact set $K\subseteq\mcO$ and every compact neighborhood $K'\supseteq K$. In contrast, for the result presented here we do \textit{not} assume any non-degeneracy condition on the noise. Moreover, the exponential rate of decay is more explicit and its relation to the decay of geometric Brownian motion becomes evident.
\begin{prop}
\label{prop:decay_due_to_ito_noise}Let $x_{0}\in L^{\infty}(\mcO)$ and let $X$ be the corresponding unique solution to \eqref{eq:BTW_SOC}. Then, $\P$-a.s. and for all $t\ge0$
\begin{equation}
X_{t}\le e^{-\mu_{t}-\td\mu t}\|x_{0}\|_{\infty},\quad\text{for a.e. }\xi\in\mcO.\label{eq:exp_decay}
\end{equation}
\end{prop}
\begin{proof}
In the following we will restrict to the case $x_{0}\not\equiv0$. For $x_{0}\equiv0$ we may proceed similarly. We consider the approximants $Y^{(\tau,\ve,\d)}$ solving \eqref{eq:tau_eps_delta-approx}. Let $M:=\|x_{0}\|_{\infty}$ and 
\[
K^{(\tau)}(t,\xi):=e^{-\td\mu^{(\tau)}(\xi)t}M+\nu t,
\]
with $\nu>0$ arbitrary, fixed. Then 
\[
\partial_{t}K^{(\tau)}=-\td\mu^{(\tau)}e^{-\td\mu^{(\tau)}t}M+\nu\ge-\td\mu^{(\tau)}K^{(\tau)}+\nu
\]
 and
\begin{align*}
e^{\mu_{t}^{(\tau)}}\D\phi^{(\tau,\ve)}(K_{t}^{(\tau)}) & =e^{\mu_{t}^{(\tau)}}\left(\ddot{\phi}{}^{(\tau,\ve)}(K_{t}^{(\tau)})|\nabla K_{t}^{(\tau)}|^{2}+\dot{\phi}{}^{(\tau,\ve)}(K_{t}^{(\tau)})\D K_{t}^{(\tau)}\right).
\end{align*}
Since $\td\mu^{(\tau)}$ is uniformly bounded we have $K^{(\tau)}\ge c>0$ for some $c>0$. We note that $\phi^{(\ve)}(r)=\sgn(r)$ for $|r|>\ve.$ Hence, also $\phi^{(\tau,\ve)}(r)=\sgn(r)$ for $|r|>\ve+\tau$ and for $\ve,\tau>0$ sufficiently small we get
\[
\ddot{\phi}{}^{(\tau,\ve)}(K^{(\tau)}),\dot{\phi}{}^{(\tau,\ve)}(K^{(\tau)})=0.
\]
Thus,
\begin{align*}
e^{\mu_{t}^{(\tau)}}\D\phi^{(\tau,\ve)}(K_{t}^{(\tau)}) & =0
\end{align*}
for all $\ve,\tau$ small enough. Moreover, we note
\[
\d e^{\mu_{t}^{(\tau)}}\D K_{t}^{(\tau)}=\d t\left(e^{\mu_{t}^{(\tau)}-\td\mu^{(\tau)}t}\D\td\mu^{(\tau)}+te^{\mu_{t}^{(\tau)}-\td\mu^{(\tau)}t}|\nabla\td\mu^{(\tau)}|^{2}\right)\le\nu,
\]
for all $t\d$ small enough. Hence,
\[
e^{\mu_{t}^{(\tau)}}\D\phi^{(\tau,\ve)}(K_{t}^{(\tau)})+\d e^{\mu_{t}^{(\tau)}}\D K_{t}^{(\tau)}-\td\mu^{(\tau)}K_{t}^{(\tau)}\le\nu-\td\mu^{(\tau)}K_{t}^{(\tau)}\le\partial_{t}K^{(\tau)}.
\]
and $K^{(\tau)}$ is a supersolution to \eqref{eq:tau_eps_delta-approx} for each $(\tau,\ve,\d)$ small enough on a time-interval $[0,T_{0}(\d)]$, where $T_{0}(\d)\uparrow\infty$ for $\d\to0$. Consequently, 
\[
Y_{t}^{(\tau,\ve,\d)}(\xi)\le K_{t}^{(\tau)}(\xi),\quad\forall(t,\xi)\in[0,T_{0}(\d)]\times\mcO.
\]
In other words, $K_{t}^{(\tau)}-Y_{t}^{(\tau,\ve,\d)}$ is a nonnegative distribution in $H^{-1}$ for all $t\in[0,T_{0}(\d)]$. Since all the limits $\tau,\ve,\d\to0$ in the construction of $Y$ hold for all $t\in[0,T]$ weakly in $H^{-1}$ and the convex cone of nonnegative distributions in $H^{-1}$ is weakly closed this implies 
\[
K_{t}-Y_{t}\ge0\quad\text{in }H^{-1}\ \text{for all }t\ge0,
\]
(using $T_{0}(\d)\uparrow\infty$ for $\d\to0$), where
\[
K(t,\xi):=e^{-\td\mu(\xi)t}M+\nu t.
\]
Since also $K_{t}-Y_{t}\in L^{2}(\mcO)$ for all $t\in[0,T]$ this implies 
\[
Y_{t}\le K_{t},\quad\text{for all }t\in[0,T],\ \text{a.e. }\xi\in\mcO.
\]
Now letting $\nu\to0$ implies the claim.\end{proof}
\begin{rem}
On an informal level the proof of Proposition \ref{prop:decay_due_to_ito_noise} relies on choosing 
\[
K(t,\xi)=e^{-\td\mu(\xi)t}\|x_{0}\|_{\infty}
\]
as a supersolution to \eqref{eq:TPME}. Since $K\ge c>0$ for some $c>0$ we have (informally)
\[
\D\sgn(K)\equiv0.
\]
Hence, the observed decay neglects the diffusive effect and is purely due to the noise and its Itô form. This explains the geometric Brownian motion type of decay in \eqref{eq:exp_decay} and is in sharp contrast to our main result Theorem \ref{thm:FTE-1} which is stable under vanishing noise (i.e. if $\td\mu\downarrow0$). 
\end{rem}
\appendix

\section{Finite time extinction for ODE}
\begin{lem}
\label{lem:ODE}Let $ $$f,g:\R_{+}\to\R_{+}$, $f$ lower semicontinuous and $g\in L_{loc}^{1}(\R_{+})$ such that there is a $K>0$ so that
\begin{equation}
f(t)\le f(s)-\int_{s}^{t}g(r)(f(r)-K)^{\a}dr,\quad\forall0\le s\le t,\label{eq:singular_ODE-1}
\end{equation}
for some $\a\in(0,1)$. Then
\[
f(t)\le\left((f(0)-K)^{1-\a}-(1-\a)\int_{0}^{t}g(r)dr\vee0\right)^{\frac{1}{1-\a}}+K,\quad\forall t\in\R_{+}.
\]
\end{lem}
\begin{proof}
We first note that by subtracting $K$ from \eqref{eq:singular_ODE-1} and replacing $f$ by $f-K$ we may suppose $K=0$.

If $f(0)=0$ then $f(t)=0$ for all $t\in[0,T]$ and nothing needs to be shown. Hence, assume $f(0)=q>0$ and let
\[
h(t):=\left(q^{1-\a}-(1-\a)\int_{0}^{t}g(r)dr\vee0\right)^{\frac{1}{1-\a}},\quad\text{for }t\in\R_{+}.
\]
Let
\begin{align*}
\tau_{1} & =\inf\{t\ge0|f(t)=0\}\\
\tau_{2} & =\inf\{t\ge0|h(t)=0\}.
\end{align*}
Since $f$ is lower semicontinuous and $h$ is continuous we have $\tau_{1},\tau_{2}>0$ and $f$, $h$ are strictly positive on $[0,\tau_{1}-\ve]$ ($[0,\tau_{2}-\ve]$ resp.) for all $\ve>0$. Thus, $h$ is the unique solution to
\begin{align}
\dot{h} & =-gh^{\a}\label{eq:singular_ODE}\\
h(0) & =q,\nonumber 
\end{align}
on $[0,\tau_{2}-\ve]$, while $f$ is a subsolution to the same equation. Since $f,h$ are strictly positive on $[0,(\tau_{1}\wedge\tau_{2})-\ve]$ comparison holds for \eqref{eq:singular_ODE} and thus
\[
f(t)\le h(t),\quad\forall t\in[0,(\tau_{1}\wedge\tau_{2})-\ve].
\]
By lower semicontinuity of $f$ we conclude $f\le h$ on $[0,\tau_{1}\wedge\tau_{2}]$. In particular, $\tau_{1}\le\tau_{2}$ and 
\[
f(t)\le h(t),\quad\forall t\in[0,\tau_{1}],
\]
which proves the claim. 
\end{proof}

\section{Some properties of Brownian Motion}
\begin{lem}
Let $\b$ be an $\R^{N}$-valued Brownian motion, $\ve>0$. Then
\[
\P[\sup_{r\in[0,t]}|\b_{r}|<\ve]>0.
\]
\end{lem}
\begin{proof}
Let $\b^{1},\b^{2}$ be independent $\R^{N}$-valued Brownian motions over the interval $[0,t]$, then $\b:=\frac{(\b^{1}\text{\textminus}\b^{2})}{\sqrt{2}}$ is also an $\R^{N}$-valued Brownian motion. There exists at least one ball $B_{\ve}(x)\subseteq C([0,t])$ such that $\mcL(\b^{1})(B_{\ve}(x))=\P[\b^{1}\in B_{\ve}(x)]=q>0$. By independence, 
\[
\P[\b^{1}\in B_{\ve}(x)\cap\b^{2}\in B_{\ve}(x)]=\P[\b^{1}\in B_{\ve}(x)]\P[\b^{2}\in B_{\ve}(x)]=q^{2}>0.
\]
Hence,
\[
\P[\b\in B_{\ve}(0)]\ge\P[\b^{1}\in B_{\ve}(x)\cap\b^{2}\in B_{\ve}(x)]\ge q^{2}>0.
\]
\end{proof}
\begin{lem}
\label{lem:constant_BM}Let $\b$ be an $\R^{N}$-valued Brownian motion. Then, there is a set $\O_{0}\subseteq\O$ of full $\P$-measure such that for all $m,n\in\N$, $\ve>0$, $\o\in\O_{0}$ there is an interval $[s,t]\subseteq[m,\infty)$ of length $|t-s|=n$ such that
\[
\sup_{r\in[s,t]}|\b_{r}(\o)-\b_{s}(\o)|<\ve.
\]
\end{lem}
\begin{proof}
Let $m,n\in\N$, $\ve>0$. We first note that by replacing $\b$ by $\b_{t}^{m}:=\b_{t+m}-\b_{m}$ we may assume $m=0$. For each $k\in\N$ we consider the interval $[kn,(k+1)n]$ and note that $\b_{t}^{k}(\o):=\b_{t+kn}(\o)-\b_{kn}(\o)=\b_{t}(\t_{kn}\o)$ are independent Brownian motions on $[0,n]$. Hence,

\[
\P[\sup_{r\in[0,n]}|\b_{r}^{k}|<\ve]=:q>0,
\]
for all $k$. We conclude
\begin{align*}
\P[\sup_{r\in[kn,k(n+1)]}|\b_{r}-\b_{kn}|\ge\ve,\ \forall k] & =\P[\sup_{r\in[0,n]}|\b_{r}^{k}|\ge\ve,\ \forall k]\\
 & =\prod_{k}\P[\sup_{r\in[0,n]}|\b_{r}^{k}|\ge\ve]\\
 & =\prod_{k}(1-q)=0.
\end{align*}
Since it is sufficient to consider $\ve\in\Q_{+}$ the claim follows.
\end{proof}

\bibliographystyle{amsalpha.bst}
\bibliography{/home/benni/cloud/current_work/latex-refs/refs}

\def\cprime{$'$}
\providecommand{\bysame}{\leavevmode\hbox to3em{\hrulefill}\thinspace}
\providecommand{\MR}{\relax\ifhmode\unskip\space\fi MR }
\providecommand{\MRhref}[2]{%
  \href{http://www.ams.org/mathscinet-getitem?mr=#1}{#2}
}
\providecommand{\href}[2]{#2}
\begin{thebibliography}{dPQRV12}

\bibitem[Att84]{A84}
H.~Attouch, \emph{Variational convergence for functions and operators},
  Applicable Mathematics Series, Pitman (Advanced Publishing Program), Boston,
  MA, 1984.

\bibitem[Bar13]{B13}
Viorel Barbu, \emph{Self-organized criticality of cellular automata model;
  absorbtion in finite-time of supercritical region into the critical one},
  Math. Methods Appl. Sci. \textbf{36} (2013), no.~13, 1726--1733.

\bibitem[BDPR09a]{BDPR09-3}
Viorel Barbu, Giuseppe Da~Prato, and Michael R{\"o}ckner, \emph{Finite time
  extinction for solutions to fast diffusion stochastic porous media
  equations}, C. R. Math. Acad. Sci. Paris \textbf{347} (2009), no.~1-2,
  81--84.

\bibitem[BDPR09b]{BDPR09-2}
\bysame, \emph{Stochastic porous media equations and self-organized
  criticality}, Comm. Math. Phys. \textbf{285} (2009), no.~3, 901--923.

\bibitem[BDPR12]{BDPR12}
\bysame, \emph{Finite time extinction of solutions to fast diffusion equations
  driven by linear multiplicative noise}, J. Math. Anal. Appl. \textbf{389}
  (2012), no.~1, 147--164.

\bibitem[BF12]{BF12}
Matteo Bonforte and Alessio Figalli, \emph{Total variation flow and sign fast
  diffusion in one dimension}, J. Differential Equations \textbf{252} (2012),
  no.~8, 4455--4480.

\bibitem[BI92]{BI92}
P{\'e}ter B{\'a}ntay and Imre~M. I{\'a}nosi, \emph{Self-organization and
  anomalous diffusion}, Phys. Rev. A \textbf{185} (1992), 11--14.

\bibitem[BR12]{BR12b}
Viorel Barbu and Michael R{\"o}ckner, \emph{Stochastic porous media equations
  and self-organized criticality: convergence to the critical state in all
  dimensions}, Comm. Math. Phys. \textbf{311} (2012), no.~2, 539--555.

\bibitem[BR13]{BR13}
Viorel Barbu and Michael Röckner, \emph{Stochastic variational inequalities and
  applications to the total variation flow perturbed by linear multiplicative
  noise}, Arch. Ration. Mech. Anal. (2013), 1--38.

\bibitem[BTW88]{BTW88}
Per Bak, Chao Tang, and Kurt Wiesenfeld, \emph{Self-organized criticality},
  Phys. Rev. A (3) \textbf{38} (1988), no.~1, 364--374.

\bibitem[CCGS90]{CCGS90}
J.~M. Carlson, J.~T. Chayes, E.~R. Grannan, and G.~H. Swindle,
  \emph{Self-organized criticality in sandpiles: nature of the critical
  phenomenon}, Phys. Rev. A (3) \textbf{42} (1990), no.~4, 2467--2470.

\bibitem[DD79]{DD79}
Gregorio D{\'{\i}}az and Ildefonso Diaz, \emph{Finite extinction time for a
  class of nonlinear parabolic equations}, Comm. Partial Differential Equations
  \textbf{4} (1979), no.~11, 1213--1231.

\bibitem[DG92]{DG92}
Albert Díaz-Guilera, \emph{{N}oise and dynamics of self-organized critical
  phenomena}, Phys. Rev. A \textbf{45} (1992), no.~12, 8551--8558.

\bibitem[DG94]{DG94}
\bysame, \emph{Dynamic renormalization group approach to self-organized
  critical phenomena}, EPL (Europhysics Letters) \textbf{26} (1994), no.~3,
  177.

\bibitem[dPQRV12]{dPQRV12}
Arturo de~Pablo, Fernando Quir{\'o}s, Ana Rodr{\'{\i}}guez, and Juan~Luis
  V{\'a}zquez, \emph{A general fractional porous medium equation}, Comm. Pure
  Appl. Math. \textbf{65} (2012), no.~9, 1242--1284.

\bibitem[GC98]{GDG90}
A.~Giacometti and A.~Chayes, Diaz-Guilera, \emph{Dynamical properties of the
  zhang model of self-organized criticality}, Phys. Rev. E \textbf{58} (1998),
  no.~1, 247--253.

\bibitem[GDG98]{GDG98}
A.~Giacometti and A.~Diaz-Guilera, \emph{Dynamical properties of the zhang
  model of self-organized criticality}, Phys. Rev. E \textbf{58} (1998), no.~1,
  247--253.

\bibitem[Ges13a]{G13-4}
Benjamin Gess, \emph{Random attractors for singular stochastic evolution
  equations}, J. Differential Equations \textbf{255} (2013), no.~3, 524--559.

\bibitem[Ges13b]{G13-2}
Benjamin Gess, \emph{Random attractors for stochastic porous media equations
  perturbed by space-time linear multiplicative noise}, to appear in: Ann.
  Probab. (2013).

\bibitem[GT13]{GT11}
Benjamin Gess and Jonas~M. T{\"o}lle, \emph{Multi-valued, singular stochastic
  evolution inclusions}, arXiv:1112.5672, to appear in J. Math. Pures Appl.
  (2013).

\bibitem[GV97]{GV97}
Victor~A. Galaktionov and Juan~L. Vazquez, \emph{Continuation of blowup
  solutions of nonlinear heat equations in several space dimensions}, Comm.
  Pure Appl. Math. \textbf{50} (1997), no.~1, 1--67.

\bibitem[Jen98]{J98}
Henrik~J. Jensen, \emph{Self-organized criticality}, Cambridge Lecture Notes in
  Physics, vol.~10, Cambridge University Press, Cambridge, 1998, Emergent
  complex behavior in physical and biological systems.

\bibitem[LSU67]{LSU67}
Olga~A. Lady\v{z}enskaja, Vsevolod~A. Solonnikov, and Nina~N. Ural'ceva,
  \emph{Linear and quasilinear equations of parabolic type}, Translated from
  the Russian by S. Smith. Translations of Mathematical Monographs, Vol. 23,
  American Mathematical Society, Providence, R.I., 1967.

\bibitem[PR07]{PR07}
Claudia Pr{\'e}v{\^o}t and Michael R{\"o}ckner, \emph{A concise course on
  stochastic partial differential equations}, Lecture Notes in Mathematics,
  vol. 1905, Springer, Berlin, 2007.

\bibitem[RW13]{RW11}
Michael R{\"o}ckner and Feng-Yu Wang, \emph{General extinction results for
  stochastic partial differential equations and applications}, J. Lond. Math.
  Soc. (2) \textbf{87} (2013), no.~2, 545--560.

\bibitem[Sho97]{S97}
Ralph~E. Showalter, \emph{Monotone operators in {B}anach space and nonlinear
  partial differential equations}, Mathematical Surveys and Monographs,
  vol.~49, American Mathematical Society, Providence, RI, 1997.

\bibitem[Tur99]{T99-2}
Donald~L. Turcotte, \emph{Self-organized criticality}, Reports on Progress in
  Physics \textbf{62} (1999), no.~10, 1377.

\bibitem[V{\'a}z06]{V06}
Juan~Luis V{\'a}zquez, \emph{Smoothing and decay estimates for nonlinear
  diffusion equations}, Oxford Lecture Series in Mathematics and its
  Applications, vol.~33, Oxford University Press, Oxford, 2006, Equations of
  porous medium type.

\bibitem[Zha89]{Z89}
Yi-Cheng Zhang, \emph{Scaling theory of self-organized criticality}, Phys. Rev.
  Lett. \textbf{63} (1989), 470--473.

\end{thebibliography}

\end{document}